\documentclass[11pt,letter]{article}
\usepackage[margin=1in]{geometry}

\usepackage{amsmath,amssymb,amsfonts,amsthm}
\usepackage{epsfig,epic,eepic}
\newtheorem{theorem}{Theorem}[section]
\newtheorem{lemma}[theorem]{Lemma}
\newtheorem{proposition}[theorem]{Proposition}
\newtheorem{corollary}[theorem]{Corollary}

\makeatletter

\@addtoreset{equation}{section}
\makeatother
\usepackage{graphicx}
\usepackage[usenames,dvipsnames]{color}
\usepackage[colorlinks=true, pdfstartview=FitV, linkcolor=blue, citecolor=blue, urlcolor=blue]{hyperref}

\newcommand{\T}{{\mathbb T}}
\newcommand{\N}{{\mathbb N}}

\newcommand{\R}{{\mathbb R}}

\newcommand{\pa}{{\partial}}

\newcommand{\eps}{{\varepsilon}}

\def\Im{{\cal I}m \,}

\def\Re{{\rm Re}  \,}

\def\mspace{\medskip \noindent}

\numberwithin{equation}{section}

\newcommand{\norm}[1]{\left\Vert #1 \right\Vert}
\newcommand{\abs}[1]{\left| #1 \right|}
\newcommand{\les}{{\; \lesssim \,}}
\newcommand{\omi}{\omega^{in}}

\newcommand{\ombl}{\omega^{bl}}

\newcommand{\vin}{v^{in}}
\newcommand{\vbl}{v^{bl}}
\newcommand{\uin}{u^{in}}
\newcommand{\ubl}{u^{bl}}

\title{Well-posedness of the hydrostatic Navier-Stokes equations}
\author{David G\'erard-Varet\footnote{IMJ-PRG, Universit\'e Paris Diderot and IUF. Email: {\textrm david.gerard-varet@imj-prg.fr}. } \qquad 
 Nader Masmoudi\footnote{Courant Institute of Mathematical Sciences, New York University. Email: {\textrm masmoudi@cims.nyu.edu}.} \qquad  Vlad Vicol\footnote{Department of Mathematics, Princeton University. Email: {\textrm vvicol@math.princeton.edu}.}}

\date{\today}

\begin{document}
\maketitle

\begin{abstract} 
We address the local well-posedness of the  {\em  hydrostatic Navier-Stokes} equations. These equations, sometimes called  {\em reduced Navier-Stokes/Prandtl},  appear as a formal limit of the Navier-Stokes system in thin domains, under certain constraints on the aspect ratio and the Reynolds number. It is known that without any structural assumption on the initial data, real-analyticity is both necessary~\cite{Renardy09} and sufficient~\cite{KukavicaTemamVicolZiane11} for the local well-posedness of the system.  In this paper we prove that for convex initial data, local well-posedness holds under simple Gevrey regularity. 
\end{abstract}

\section{Introduction}
The present paper is devoted to the study of the following two-dimensional system: 
\begin{subequations}
\label{eq:HNS}
\begin{align}
\pa_t u + u \pa_x u + v \pa_y u  +  \pa_x p -  \eta \pa^2_y u & = 0, \quad (x,y) \in \T \times (0,1), 
\label{eq:HNS:evo}\\ 
\pa_y p & = 0, \quad (x,y) \in \T \times (0,1), \\
\pa_x u + \pa_y v & = 0, \quad (x,y) \in \T \times (0,1), \label{eq:HNS:incompressible}\\
u\vert_{y=0,1} = v\vert_{y=0,1} & = 0, \quad x \in \T, \label{eq:HNS:BC}
\end{align}
\end{subequations}
where $\eta > 0$. The unknowns of this system  are $(u,v) = (u,v)(x,y,t)$ and  $p = p(x,y,t)$, which model respectively the velocity field  and pressure of a fluid flow. The boundary condition \eqref{eq:HNS:BC} corresponds to a no-slip condition at the walls $y=0,1$. With respect to the tangential variable $x$ we impose $\T$-periodic (lateral) boundary conditions.

\mspace
Note that upon integrating in $y$ the incompressibility equation \eqref{eq:HNS:incompressible}, using the boundary condition for $v$ \eqref{eq:HNS:BC} we obtain the compatibility condition
\begin{align}
\pa_x \int_0^1 u(x,y,t)  dy =0 
\label{eq:u:compatibility}
\end{align}
for all $x\in \T$ and $t\geq 0$, so that the vertical mean of $u$ is just a function of time. Condition~\eqref{eq:u:compatibility} allows us to compute the pressure gradient, cf.~\eqref{expp} below, and to obtain the boundary condition for the vorticity, cf.~\eqref{eq:HNS:vort:BC} below.

\mspace
System \eqref{eq:HNS} is  formally obtained~\cite{LagreeLorthois05,Renardy09} when considering the asymptotics of the two-dimensional Navier-Stokes in a thin domain:  $\Omega = (0,L) \times (0,l)$ with $\delta = \frac{l}{L} \ll 1$. After a proper rescaling 
$$t := \frac{Ut}{L}, \quad  x := \frac{x}{L}, \quad y :=  \frac{y}{l}, \quad u := \frac{u}{U}, \quad v := \frac{v}{\delta U},  $$
the Navier-Stokes equation becomes
\begin{subequations}  
\label{NS}
\begin{align}
\pa_t u + u \pa_x  u + v \pa_y u + \pa_x p  - \eta \delta^2 \pa^2_x  -  \eta \pa^2_y u & = 0, \quad (x,y) \in \T \times (0,1), \\ 
\delta^2 (\pa_t v + u \pa_x  v + v \pa_y v) + \pa_y p - \eta \delta^4 \pa^2_x v  -  \eta \delta^2 \pa^2_y v & = 0, \quad (x,y) \in \T \times (0,1), \\ 
 \pa_x u + \pa_y v & = 0, \quad (x,y) \in \T \times (0,1),
\end{align}
\end{subequations}
where $\eta = \frac{1}{\delta^2 \textrm{Re}}$, with $\textrm{Re} = \frac{U L}{\nu}$ the Reynolds number. 
If we assume  $\eta \sim 1$ and keep the leading order terms as $\delta \to 0$, or if we assume $\eta \ll 1$ and keep both the leading order   and next order terms in \eqref{NS}, we end up with  \eqref{eq:HNS}. 

\mspace
Our concern here will be the local in time well-posedness of \eqref{eq:HNS}. Besides its mathematical relevance, this problem is meaningful from the point of view of hydrodynamic stability, notably with regards to the properties of the so-called {\em primitive equations}: 
\begin{subequations}  \label{PE}
\begin{align}
\pa_t u + u \pa_x  u + v \pa_y u + \pa_x p  - \eta' \pa^2_x  -  \eta \pa^2_y u & = 0, \quad (x,y) \in \T \times (0,1), \\ 
 \pa_y p  & = 0, \quad (x,y) \in \T \times (0,1), \\ 
 \pa_x u + \pa_y v & = 0, \quad (x,y) \in \T \times (0,1).  
\end{align}
\end{subequations}
This model and  its three-dimensional counterpart  are very important in atmospheric sciences, after accounting for gravity and many other features~\cite{LionsTemamWang92a,LionsTemamWang92b,TemamZiane04,PetcuTemamZiane09}.  For positive values of tangential and transverse viscosity coefficients, they are known to be globally well-posed in the Sobolev setting in both the two and the three dimensional case~\cite{Ziane97b,BreschGuillenMadmoudiRodriguez03,BreschKazhikhovLemoine04,TemamZiane04,CaoTiti07, Kobelkov07,KukavicaZiane07,KukavicaZiane08}, and the vanishing viscosity limit $\eta,\eta'\to 0$ can be characterized in the real-analytic category~\cite{KukavicaLombardoSammartino16}. Yet, in the absence of additional  turbulent viscosity, the dimensional analysis of \eqref{NS} shows that the tangential diffusion coefficient $\eta'$ is expected to be very small. This allows to relate the well/ill-posedness of \eqref{eq:HNS} and the 
stability/instability properties of \eqref{PE}. For instance,  assume that \eqref{eq:HNS} is linearly ill-posed without analyticity in $x$: a result in this direction was shown in \cite{Renardy09}, and will be discussed later on. It roughly means that, at least in the early stages of the evolution, there are perturbations with wave number $k \gg 1$ in $x$ that grow like $e^{|k| t}$. From there, if $\eta'$ is small enough so that  $\eta' |k|^2 \ll 1$, one can expect  the tangential diffusion $- \eta' \pa^2_x$  to  stay negligible, and the perturbation to be an approximate solution of \eqref{PE} (with Dirichlet conditions). This can result in a growth almost as strong as  $e^{t/\sqrt{\eta'}}$, showing the strong instability of \eqref{PE}. We note that if one keeps $\eta'>0$ in \eqref{PE} while setting $\eta=0$,  the local well-posedness can be established for Sobolev initial datum~\cite{CaoLiTiti16,CaoLiTiti17}, confirming that the horizontal dissipation dominated equation is much more stable that the hydrostatic Navier-Stokes system~\eqref{eq:HNS} considered in this paper.

\mspace
From a mathematical perspective, system \eqref{NS} is  reminiscent of the two-dimensional Prandtl system, describing boundary layer flows. The latter is set in a half-plane, say $\T \times \R_+$, and reads
\begin{subequations} \label{P}
\begin{align}
\pa_t u + u \pa_x u + v \pa_y u  +  \pa_x p -  \eta \pa^2_y u & = 0, \quad (x,y) \in \T \times \R_+, \\ 
\pa_y p & = 0, \quad (x,y) \in \T \times \R_+, \\
\pa_x u + \pa_y v & = 0, \quad (x,y) \in \T \times \R_+, \\
u\vert_{y=0} = v\vert_{y=0} & = 0, \\
\lim_{y \rightarrow +\infty} u = u^\infty, \lim_{y \rightarrow +\infty} p & = p^\infty. 
\end{align}
\end{subequations}
Hence, the only difference with \eqref{eq:HNS} lies in the domain and in the boundary conditions. Here, $u^\infty$ and $p^\infty$ are given data, related to the Euler flow above the boundary layer. In particular, as $p$ does not depend on $y$, it is no longer an unknown of the system. This is a major difference with \eqref{eq:HNS}, where $p$ can be seen as a Lagrange multiplier, associated to the constraint that $v = - \int_0^y \pa_x u$ vanishes at $y=1$ (see \eqref{expp} below).  

\mspace
The well-posedness properties of \eqref{P} are now well-understood, and depend on the monotonicity properties of the initial data. Roughly, if the  data have Sobolev regularity, and if furthermore the initial data are monotonic in $y$, \eqref{P} has local in time Sobolev solutions~\cite{Oleinik66,MasmoudiWong15}. On the other hand, without monotonicity, system \eqref{P} is ill-posed in Sobolev spaces \cite{GerardVaretDormy10,GerardVaretNguyen12}. Local in time well-posedness can be achieved when the initial datum is real analytic~\cite{SammartinoCaflisch98a, KukavicaVicol13}, and even under the milder condition of Gevrey regularity in $x$~\cite{GerardVaretMasmoudi13}. We refer to \cite{EEngquist97,XinZhang04,GerardVaretMaekawaMasmoudi16,IgnatovaVicol16,KukavicaVicolWang17,DalibardMasmoudi18} and references therein for more results on the Prandtl system such as singularities, long time behavior, and Gevrey-class stability.  Interestingly, the instability mechanism that yields ill-posedness in Sobolev involves in a crucial manner the lack of monotonicity and  the diffusion term $-\eta \pa^2_y u$. Indeed, the inviscid version of Prandtl, that is 
 \begin{subequations} \label{IP}
\begin{align}
\pa_t u + u \pa_x u + v \pa_y u  +  \pa_x p  & = 0, \quad (x,y) \in \T \times \R_+, \\ 
\pa_y p & = 0, \quad (x,y) \in \T \times \R_+, \\
\pa_x u + \pa_y v & = 0, \quad (x,y) \in \T \times \R_+, \\
v\vert_{y=0} & = 0, \\
\lim_{y \rightarrow +\infty} p & = p^\infty,
\end{align}
\end{subequations}
has local smooth solutions for smooth data, as can be shown by the method of characteristics \cite{HongHunter03}.  

\mspace
With regards to this recent understading of the Prandtl system, it is very natural to ask about the local well-posedness of \eqref{eq:HNS}, and to start from the consideration of the inviscid case $\eta = 0$, namely
\begin{subequations} \label{HE}
\begin{align}
\pa_t u + u \pa_x u + v \pa_y u  +  \pa_x p  & = 0, \quad (x,y) \in \T \times (0,1), \\ 
\pa_y p & = 0, \quad (x,y) \in \T \times (0,1), \\
\pa_x u + \pa_y v & = 0, \quad (x,y) \in \T \times (0,1), \\
v\vert_{y=0,1} & = 0.  
\end{align}
\end{subequations}
This {\em hydrostatic Euler system} has been the matter of many studies \cite{Brenier99,Grenier99,Brenier03,Renardy09,KukavicaTemamVicolZiane11,MasmoudiWong12,KukavicaMasmoudiVicolWong14,CaoIbrahimNakanishiTiti15,Wong15}. Contrary to \eqref{IP}, existence of local strong solutions requires a structural assumption, namely the  uniform convexity (or concavity) in variable $y$ of the initial data. {\it A contrario}, the presence of inflexion point may trigger high-frequency instability. This point was established in  article \cite{Renardy09}. The author considers in \cite{Renardy09} the linearization of \eqref{HE} around shear flows $u = U_s(y), v = 0$. More precisely, he shows that if the equation $ \int_0^{1} (U_s(y) - c)^{-2} dy = 0$ has complex roots, then the linearized hydrostatic Euler system admits perturbations which have  wavenumber $k$ in $x$ and  grow like $e^{\delta k t}$, $\delta > 0$, for all $k \gg 1$. Back to the nonlinear problem \eqref{HE}, one can only expect to show short time stability for data whose Fourier transform in $x$ behaves like $e^{-\delta |k|}$ for large $k$. This corresponds to analytic data in $x$. Local well-posedness in the analytic setting was established in \cite{KukavicaTemamVicolZiane11}. Moreover, it is mentioned in \cite{Renardy09} that this high-frequency instability persists in the case of the viscous system \eqref{eq:HNS}, at least for small enough $\eta$. 

\mspace
Considering all these results,  the remaining task is to analyse the  viscous system \eqref{eq:HNS} for convex (or concave) initial data. This is the purpose of this paper.  It raises strong mathematical issues, related to the control of $x$ derivatives of the solution. In particular, we find
$$
\pa_t (\pa_x u) +  (u \pa_x + v \pa_y)  (\pa_x u)  + (\pa_x u)^2  + (\pa_x v) \pa_y u +  \pa_x (\pa_x p) -  \eta \pa^2_y (\pa_x u)  = 0. 
$$
One of the main problems in controlling $\pa_x u$ is the term $\pa_x v \pa_y u$. Indeed, $\pa_x v = - \int_0^y \pa_x^2 u$ is recovered from the divergence-free condition, so that it can be seen as a first oder operator in $x$ applied to $\pa_x u$. As this first order term has no skew-symmetry, it does not disappear from energy estimates, so that standard energy arguments can only be conclusive with the help of analyticity. 
In the case of the hydrostatic Euler system, the way out of this difficulty consists in considering the (approximate) vorticity $\omega = \pa_y u$. 
Its tangential derivative is seen to satisfy
$$ \pa_t (\pa_x \omega) \: + \:  (u \pa_x  + v \pa_y) (\pa_x \omega)  \: + \: (\pa_x u) \, (\pa_x \omega) \: + \: (\pa_x v) \, \pa_y \omega  =  0.  
$$
Under a uniform convexity or concavity assumption $|\pa_y \omega| \ge \alpha$, the idea is to test the equation against $\pa_x \omega/ \pa_y \omega$ rather than $\pa_x \omega$,  to take advantage of the cancellation: 
$$ \int  \pa_x v \, \pa_x \omega = - \int \pa_y \pa_x v \, \pa_x u  =  \int \pa_x^2 u \, \pa_x u = 0. $$
This allows to get rid of the bad term, and is the starting point of the local well-posedness argument. Such an idea was used previously in~\cite{Grenier,MasmoudiWong12}. 

\mspace
Unfortunately, this manipulation, that we will call {\em the hydrostatic trick}, is not fully appropriate to the viscous system \eqref{eq:HNS}. The reason is that in the estimate for $\pa_x \omega$,  the viscous term generates  extra boundary integrals such as
$$I^\flat = \eta \int_{\T \times \{0\}} \pa_y \pa_x \omega \frac{\pa_x \omega}{\pa_y \omega} dx, \quad    I^\sharp= \eta \int_{\T \times \{1\}} \pa_y \pa_x \omega \frac{\pa_x \omega}{\pa_y \omega} dx. 
$$
The value of  $\pa_y \pa_x \omega$ at the boundary can be obtained from the equation on $\pa_x u$, and yields for instance (the computation will be detailed later)
$$ \pa_y \pa_x  \omega\vert_{y=0} =  \pa_x^2 p = - 2 \pa_x \int_0^1  u \,  \pa_x u \, dy \: + \: \pa_x \omega\vert_{y=1}  - \pa_x \omega\vert_{y=0}. 
$$ 
The issue comes from the first term at the right hand-side, which is again a first order term in $\pa_x u$ without any skew-symmetric structure. In other words, {\em there is an additional loss of derivative compared to the Prandtl equation},  so that obtaining well-posedness below analytic regularity is  challenging. This is our goal in what follows, and we prove in Theorem~\ref{thm:main} below the  local well-posedness under Gevrey regularity of class $9/8$ in the $x$ variable, under an extra convexity assumption  in $y$. 

\section{Main result and strategy}
For notational simplicity, from now one we will set $\eta = 1$ in \eqref{eq:HNS}.
 Let $\Omega = \T \times (0,1)$. For $\tau > 0$, $\gamma \ge 1$, we define the Gevrey norm
$$ \| f \|_{\gamma,\tau}^2 =  \sum_{j=0}^\infty \tau^{2j} (j!)^{-2\gamma} \| \pa_x^j f \|_{L^2(\Omega)}^2  . $$
 Functions $f$ satisfying  $\| f \|_{\gamma,\tau} < +\infty$ are in Gevrey class $\gamma$ with respect to $x$,  measured in $L^2$ in variable $y$. 
 Our main result is the following:
\begin{theorem}[\bf Well-posedness for convex Gevrey-class initial datum]
\label{thm:main}
Let $\tau^0 > \tau_1 > 0$, $\gamma \le 9/8$. Let $u_0$ a function satisfying the regularity condition 
\begin{align}
\| \pa_y u_0 \|_{\gamma,\tau^0} +   \| \pa_y^3 u_0 \|_{\gamma, \tau^0}  <  +\infty,
\label{eq:thm:main:1}
\end{align}
the convexity condition 
\begin{align}
\inf_{\Omega}  \pa^2_y u_0 > 0,
\label{eq:thm:main:2}
\end{align}
and the compatibility conditions  $\pa_x \int_0^1 u_0 dy = 0$, $u_0\vert_{y=0,1} = 0$, 
$$\pa^2_y u_0\vert_{y=0,1} =  \int_0^1 (-\pa_x u_0^2 + \pa^2_y u_0 ) dy \: - \: \int_\Omega \pa^2_y u_0.$$   
Then there exists  $T > 0$, and a unique solution $u$ of \eqref{eq:HNS} with initial data $u_0$ that satisfies 
\begin{equation*}
\sup_{t \in [0,T]} \left(\| \pa_y u(t) \|_{\gamma,\tau_1} + \|  \pa^3_y u(t) \|_{\gamma,\tau_1}\right)  < +\infty. 
\label{eq:thm:main:3}
\end{equation*}
and 
\begin{align}
\inf_{t \in [0,T] \times \Omega} \pa^2_y u > 0 .
\label{eq:thm:main:4}
\end{align}
\end{theorem}
\noindent
A few remarks are in order: 
\begin{itemize}
\item The main point in our result is that we prove local well-posedness without analyticity, reaching exponents $\gamma > 1$. The value 
$\gamma = 9/8$ is due to technical limitations, and could certainly be improved. The optimal value that can be expected for $\gamma$, or even the possibility of well-posedness in the Sobolev setting are interesting open questions. Our conjecture - based on a formal  parallel with Tollmien-Schlichting instabilities for Navier-Stokes  \cite{GrGuNg} - is that the best exponent  possible should be $\gamma = 3/2$, but such result is for the time being out of reach.  If confirmed, it would emphasize the destabilizing role of viscosity. 

\item We  loose  on  the radius $\tau$ of Gevrey regularity, going from $\tau^0$ to $\tau_1$ in positive time. This loss is very standard \cite{SammartinoCaflisch98a,KukavicaTemamVicolZiane11,KukavicaVicol13,GerardVaretMasmoudi13}.  
\item Besides the Gevrey regularity assumption~\eqref{eq:thm:main:1}, the key assumption is $\inf_\Omega  \pa^2_y u_0 > 0$, which corresponds to a strictly convex initial data. The strict concavity condition
$\sup_\Omega  \pa^2_y u_0 < 0$ would work as well. On the opposite, as discussed before, we do not expect such well-posedness to hold for data with inflexion points~\cite{Renardy09}.
\item The first compatibility condition $\pa_x \int_0^1 u_0 = 0$ is here to ensure that  \eqref{eq:u:compatibility} holds for all time. Note that we can use \eqref{eq:u:compatibility} to determine $\pa_x p$: applying $\pa_x$ to \eqref{eq:HNS:evo}, taking the mean over $y \in (0,1)$, integrating by parts in the term $\int_0^1 v\pa_y u\, dy$, and using the periodic lateral boundary conditions, we find:
\begin{equation} \label{expp}
\pa_x p 
= \tilde \omega\vert_{y=1} - \tilde \omega\vert_{y=0} - \pa_x  \int_0^1 u^2 dy, \qquad x \in \T,
\end{equation}
where $\omega = \pa_y u$ is the vorticity, and we have denoted by
\begin{align}
 \tilde \omega(x,y,t) = \omega(x,y,t) - \int_{\T} \omega(x,y,t) dx, \qquad y \in \{0,1\},
 \label{eq:tilde:omega}
\end{align}
the zero mean (in $x$) boundary vorticity. We will use the notation \eqref{eq:tilde:omega} throughout the paper. Note that for $y \in \{0,1\}$, the functions $\omega$ and $\tilde \omega$ only differ by a function of time.
\item The second and third compatibility conditions  can be explained as follows. Most of our analysis relies on the control of the vorticity $\omega = \pa_y u$. We notably need some  bound on  $\sup_{t \in [0,T]} \| \omega\|_{\gamma,\tau}$ for $\tau \in [\tau_1,\tau^0)$ . If we leave aside the Gevrey regularity in $x$,  this corresponds to an  $L^\infty_t H^1_y$ bound on $u$. As $u$ satisfies a heat type equation with Dirichlet condition,  it is well-known that such an $L^\infty_t H^1_y$ bound requires the compatibility condition $u\vert_{t=0}\vert_{y=0,1} =  u\vert_{y=0,1} \vert_{t=0}$. In view of \eqref{eq:HNS:incompressible}, this amounts to the second compatibility condition of the theorem: $u_0\vert_{y=0,1} = 0$. 
 
 \mspace
Similarly,  the last compatibility condition is  related to the fact that we need a bound for   $\sup_{t \in [0,T]} \| \pa_t \omega\|_{\gamma, \tau}$ for $\tau \in [\tau_1,\tau^0)$.  More precisely, this condition can be derived from the system obeyed by $\omega = \pa_y u$, which is: 
 \begin{subequations}
\label{eq:HNS:vort}
\begin{align}
\pa_t \omega + u \pa_x \omega + v \pa_y \omega  -   \pa^2_y \omega  = 0, \quad (x,y) \in \T \times (0,1), &
\label{eq:HNS:vort:evo}\\
\pa_y \omega\vert_{y=0,1}   = \tilde \omega\vert_{y=1} - \tilde \omega\vert_{y=0} - \pa_x  \int_0^1 u^2 dy. &  \label{eq:HNS:vort:BC}
\end{align}
\end{subequations}
Indeed, \eqref{eq:HNS:vort:evo} follows from differentiating \eqref{eq:HNS:evo} in $y$, while the boundary condition \eqref{eq:HNS:vort:BC}  is obtained by evaluating \eqref{eq:HNS:evo} at $y=0,1$,  using the Dirichlet boundary conditions for $u$ and $v$ in \eqref{eq:HNS:BC}, and the formula for the pressure gradient \eqref{expp}. Now, from \eqref{eq:HNS:vort:evo}, it appears that an  $L^\infty_t L^2_y$ control of   $\pa_t \omega$ is similar to an $L^\infty_t L^2_y$ control of $\pa^2_y \omega$, meaning a $L^\infty_t H^1_y$ control of  $\pa_y \omega$. By differentiating \eqref{eq:HNS:vort:evo}, one sees that $\pa_y \omega$ satisfies a heat like equation, and by \eqref{eq:HNS:vort:evo}, it also satisfies  a Dirichlet  type condition. Again, an $L^\infty_t H^1_y$ control requires  $\pa_y \omega\vert_{t=0}\vert_{y=0,1} = \pa_y\omega\vert_{y=0,1}\vert_{t=0}$, which by \eqref{eq:HNS:vort:BC} amounts to  the third compatiblity condition.  
\end{itemize}

\mspace
{\bf General strategy of the proof.} Our analysis is based on the vorticity evolution \eqref{eq:HNS:vort}. We want to benefit from the so-called hydrostatic trick, which consists in establishing $L^2$ estimates for the weighted derivatives $\pa_x^j \omega/\sqrt{\pa_y \omega}$. The difficulty is that these estimates are not compatible with the diffusion  $-\pa^2_y \omega$, which creates boundary terms involving $\pa^j_x \pa_y \omega\vert_{y=0}$. Because of the extra $x$-derivative at the right-hand side of  \eqref{eq:HNS:vort:BC}, one can not close an estimate at the Sobolev level. 

\mspace
To overcome this difficulty, our first idea is to write  $\omega = \omega^{in} + \omega^{bl}$, where $\omega^{bl}$ is a boundary corrector which  solves (approximately): 
\begin{equation*} 
\pa_t \omega^{bl} - \pa^2_y \omega^{bl} = 0, \quad \pa_y \omega^{bl}\vert_{y=0,1} =  - \pa_x  \int_0^1 u^2 dy, 
\end{equation*}
where the right side of the Neumann boundary condition is seen as a given data. With this splitting, the bad term is removed from the Neumann condition on  $\omega^{in}$, so that we may apply the hydrostatic trick to this quantity.  Still, this approach is obviously not enough: the equation for $\omega^{in}$ still  involves $\omega$, either directly or through $\omega^{bl}$, so that no closed estimate is available on  $\omega^{in}$. 

\mspace
This is where we shall take advantage of Gevrey regularity. To explain this point, it is simpler to consider the linearization of \eqref{eq:HNS:vort} around a  shear flow $u = (u_s(y), 0)$: 
\begin{align*}
 \pa_t \omega + u_s \pa_x \omega + u_s'' v - \pa^2_y \omega  = 0, \quad \pa_x u + \pa_y v = 0, \quad 
 \pa_y \omega\vert_{y=0,1}  = \tilde \omega\vert_{y=1}  - \tilde \omega\vert_{y=0} - 2 \pa_x  \int_0^1 u_s u dy. 
\end{align*}
As this system has $x$-independent coefficients, one can Fourier transform in $x$. More precisely, looking for local well-posedness in Gevrey class 
$\gamma$, it is natural to look for solutions in the form $ \omega = e^{k^{1/\gamma} t} e^{i k x} \hat{\omega}_k(t,y)$. We end up with the following system for the boundary layer corrector: 
\begin{equation*} 
(k^{1/\gamma} +  \pa_t) \hat{\omega}^{bl} - \pa^2_y \hat{\omega}^{bl} = 0, \quad \pa_y \hat{\omega}_k^{bl}\vert_{y=0,1} =  - 2 i k   \int_0^1 u_s \hat{u}_k dy.
\end{equation*}
Explicit calculations on this system reveal  that Gevrey regularity in $x$  is converted into spatial localization in $y$: for $k \gg 1$, $\hat{\omega}^{bl}_k$ has a boundary layer behaviour, with concentration near $y=0,1$ at scale $k^{-\frac{1}{2\gamma}}$. Roughly, neglecting the upper boundary, one can think of 
\begin{align*}
 \hat{\omega}^{bl}_k &\approx k^{1-\frac{1}{2\gamma}} W(t,k^{\frac{1}{2\gamma}} y) \int_0^1 u_s \hat{u}_k dy, \\
 \hat{u}^{bl}_k &\approx k^{1-\frac{1}{\gamma}} U(t,k^{\frac{1}{2\gamma}} y) \int_0^1 u_s \hat{u}_k dy. 
 \end{align*}
Now, the idea is to write 
$$ \int_0^1 u_s \hat{u}_k dy = \int_0^1 u_s  \hat{u}^{bl}_k  \: +  \:    \int_0^1 u_s  \hat{u}^{in}_k  =  \left( k^{1-\frac{1}{\gamma}}  \int_0^1 u_s(y) U(t,k^{\frac{1}{2\gamma}} y) dy \right)  \int_0^1 u_s \hat{u}_k dy \: + \:   \int_0^1 u_s  \hat{u}^{in}_k.  $$
In short, one can check that for $\gamma \le 2$,  we have $ k^{1-\frac{1}{\gamma}}  \int_0^1 u_s(y) U(t,k^{\frac{1}{2\gamma}} y) dy  = o(1)$ in the limit of large $k$, so that the first term at the right-hand side can be absorbed in the left-hand side. This leads to a control of $\int_0^1 u_s u $, and thus  of $\omega^{bl}$, in terms of $\omega^{in}$. From there,  one can get closed estimates on $\omega^{in}$. 

\mspace
Of course, this strategy is made more difficult when dealing with the $x$-dependent and  nonlinear system \eqref{eq:HNS:vort}. In particular, the Fourier approach is no longer convenient, and we must use the characterization of Gevrey regularity  in the physical space, through the family $\{ \pa^j_x \omega\}_{j\in \N}$. In order to take advantage of the boundary layer phenomenon,  we shall introduce Gevrey norms with extra-weight $(j+1)^r$, see \eqref{eq:Mj:def}. The boundary layer phenomenon will be reflected by the fact that multiplication by  $y$ or integration in $y$ will generate a gain in the exponent $r$, see Lemma~\ref{lem:BL:1}. Such gain will make possible the control of boundary layer quantities by $\omega^{in}$, cf.~Lemma~\ref{lem:h:omega:in}. 

\mspace
From there, the analysis will focus on weighted estimates for $\omega^{in}$, using the hydrostatick trick. As usual in nonlinear problems, these estimates will be obtained conditionally to certain bounds (notably a lower bound on $\pa_y \omega$, to benefit from convexity). We will show that such bounds are preserved in small time, which will require estimates on the time derivative $\pa_t \omega$, as well as maximum principle arguments for $\pa_y \omega$.

\section{Preliminaries} \label{sec:preliminaries}

 As usual in this kind of analysis, we will focus on {\it a priori} estimates. This means that from Section \ref{sec:preliminaries} to Section \ref{sec:max}, we will assume implicitly that we  already have a solution of \eqref{eq:HNS} on $[0,T]$ with all necessary smoothness, and we will collect properties and estimates about this solution. Only in Section \ref{sec:proof:main} will we describe the way of constructing solutions.

\subsection{Norms and notation}
Let $\gamma \ge 1$, $r \in \R$, $\tau > 0$.  We introduce a refined two-dimensional Gevrey norm 
\begin{align}
\norm{f}_{\gamma,r,\tau}^2 = \sum_{j\geq 0} M_j^2   \norm{\pa_x^j f}_{L^2_{x,y}(\T \times [0,1])}^2 ,\qquad \mbox{where} \qquad M_j = \frac{(j+1)^r \tau^{j+1}}{(j!)^\gamma}.
\label{eq:Mj:def}
\end{align}
Note that the $L^2$ norm in space is only used on $\Omega=  \T \times[0,1]$, although the functions may be defined on the half-space $\T \times [0,\infty)$. We note that if $r' \geq r$ then $\norm{\cdot}_{\gamma,r',\tau} \geq \norm{\cdot}_{\gamma,r,\tau}$. 

\mspace
For functions which are independent of the $y$ variable, we use the one-dimensional counterpart
\begin{align*}
\abs{f}_{\gamma,r,\tau}^2 = \sum_{j\geq 0} M_j^2 \norm{\pa_x^j f}_{L^2_{x}(\T)}^2 ,
\end{align*}
where $M_j$ is defined as before. Similarly, if $r' \geq r$ then $\abs{\cdot}_{\gamma,r',\tau} \geq \abs{\cdot}_{\gamma,r,\tau}$. 

\mspace
Let $\tau^0$, $\tau_1$ as in the theorem, and let  $\tau_0$ such that $\tau^0 > \tau_0 > \tau_1$. Throughout the paper, the Gevrey-class radius $\tau$ will be defined by
\begin{align}
\tau(t) = \tau_0 \exp(-\beta t), \quad 
\label{eq:tau:def}
\end{align}
where $\beta\geq 1$, $t \in [0,T]$, and $T$  always small enough so that $\tau(t) \ge \tau_1$. In particular $\dot \tau(t) = - \beta \tau(t)$. 

\mspace
We will use $a \les b$ to denote the existence of a constant $C>0$, which may depend only on $\gamma,\tau_0,\tau_1$, and $r$, such that $a \leq C b$. Similarly, will use $a \ll b$ to denote the existence of a sufficiently large constant $C>0$, which may depend only on $\gamma,\tau_0,\tau_1$, and $r$, such that $C a \leq b$.

\mspace
For any function $f$ we use the notation
\begin{align}
 f_j &= M_j \pa_x^j f
 \label{eq:subindex:j}
\end{align}
where $M_j$ is defined in~\eqref{eq:Mj:def} and depends on $r,\gamma$, and $\tau$. With this notation we have
\begin{align*}
\norm{f}_{\gamma,r,\tau}^2 = \sum_{j\geq 0} \norm{f_j}_{L^2_{x,y}}^2 \qquad \mbox{and} \qquad 
\abs{f}_{\gamma,r,\tau}^2 = \sum_{j\geq 0} \norm{f_j}_{L^2_{x}}^2.
\end{align*}

\subsection{A boundary layer lift}
The boundary condition \eqref{eq:HNS:vort:BC} in the vorticity evolution \eqref{eq:HNS:vort} motivates the introduction of a boundary layer lift for the the vorticity, which we describe next. Throughout the paper we appeal to Gevrey estimates for the system
\begin{subequations}
\label{eq:flat}
\begin{align}
(\pa_t - \pa_y^2) \omega^{\flat} &= 0 \\
(\pa_y\omega^{\flat} + 2 \omega^{\flat})\vert_{y=0} &= \pa_x h\vert_{y=0}\\
\omega^{\flat}\vert_{t=0} &=0 \label{eq:flat:c}
\end{align}
\end{subequations}
posed for $t \in [0,T]$, $x\in \T$, and $y \in \R_+$. Here $h$ is a placeholder for $- \left( \int_0^1 u^2\, dy - \int_\T\int_0^1 u^2\, dy dx\right)$. Since the boundary datum for $\omega^\flat$ is a pure $x$ derivative (and this is the only nontrivial datum), we note that \eqref{eq:flat} immediately implies that $\int_\T \omega^\flat (x,y,t) dx = 0$, for any $y\geq 0$. We also define 
\begin{align}
 u^{\flat}(x,y) &= \int_{+\infty}^y \omega^\flat(x,z) dz \\
 v^{\flat}(x,y) &= \int^{+\infty}_y \pa_x u^\flat(x,z) dz. \label{def:vflat}
 \end{align}

\begin{lemma}
\label{lem:BL:1}
Let  $r \in \R$, $\beta \ge 1$ and   $T>0$ such that $\tau(t) \geq \tau_1$ for $t \in [0,T]$. The boundary layer vorticity $\omega^\flat$ obeys
\begin{subequations}
\begin{align}
 \int_0^t \norm{\omega^\flat(s)}_{\gamma,r,\tau(s)}^2 ds  &\les \frac{1}{\beta^{3/2}} \int_0^t \abs{h(s)}_{\gamma,r+\gamma-\frac{3}{4},\tau(s)}^2 ds \label{eq:omega:flat:Gevrey} \\
  \int_0^t \norm{y \, \omega^\flat(s)}_{\gamma,r,\tau(s)}^2 ds  &\les \frac{1}{\beta^{5/2}}  \int_0^t \abs{h(s)}_{\gamma,r+\gamma-\frac{5}{4},\tau(s)}^2 ds \label{eq:y:omega:flat:Gevrey}\\
\int_0^t \norm{\pa_y  \omega^\flat(s)}_{\gamma,r,\tau(s)}^2 ds  &\les \frac{1}{\beta^{1/2}}  \int_0^t \abs{h(s)}_{\gamma,r+\gamma-\frac{1}{4},\tau(s)}^2 ds \label{eq:dy:omega:flat:Gevrey}\\
    \int_0^t \norm{y \pa_y \omega^\flat(s)}_{\gamma,r,\tau(s)}^2 ds  &\les \frac{1}{\beta^{3/2}}  \int_0^t \abs{h(s)}_{\gamma,r+\gamma-\frac{3}{4},\tau(s)}^2 ds 
    \label{eq:y:dy:omega:flat:Gevrey}\\
       \int_0^t \abs{\omega^\flat(s)\vert_{y=1}}_{\gamma,r,\tau(s)}^2 ds  &\les \frac{1}{\beta^{20}} \int_0^t \abs{h(s)}_{\gamma,r+\gamma-10,\tau(s)}^2 ds 
   \label{eq:omega:flat:y=1}\\
  \int_0^t \abs{\pa_y \omega^\flat(s)\vert_{y=1}}_{\gamma,r,\tau(s)}^2 ds  &\les \frac{1}{\beta^{20}}  \int_0^t \abs{h(s)}_{\gamma,r+\gamma-10,\tau(s)}^2 ds, 
  \label{eq:dy:omega:flat:y=1}
  \end{align}
  \end{subequations}
the boundary layer velocity $u^\flat$ obeys
  \begin{subequations}
  \begin{align}
 \int_0^t \norm{u^\flat(s)}_{\gamma,r,\tau(s)}^2 ds  &\les \frac{1}{\beta^{5/2}}  \int_0^t \abs{h(s)}_{\gamma,r+\gamma-\frac{5}{4},\tau(s)}^2 ds 
 \label{eq:u:flat:Gevrey} \\
  \int_0^t \norm{y u^\flat(s)}_{\gamma,r,\tau(s)}^2 ds  &\les \frac{1}{\beta^{7/2}}  \int_0^t \abs{h(s)}_{\gamma,r+\gamma-\frac{7}{4},\tau(s)}^2 ds ,
  \label{eq:y:u:flat:Gevrey}\\
  \int_0^t \abs{ u^\flat(s)\vert_{y=\frac{1}{2}}}_{\gamma,r,\tau(s)}^2 ds  &\les \frac{1}{\beta^{20}}  \int_0^t \abs{h(s)}_{\gamma,r+\gamma-10,\tau(s)}^2 ds, 
  \label{eq:u:flat:y=12}
  \end{align}
  \end{subequations}
  and the boundary layer velocity $v^\flat$ satisfies
  \begin{subequations}
  \begin{align}
 \int_0^t \norm{v^\flat(s)}_{\gamma,r,\tau(s)}^2 ds  &\les \frac{1}{\beta^{7/2}}  \int_0^t \abs{h(s)}_{\gamma,r+2\gamma-\frac{7}{4},\tau(s)}^2 ds 
 \label{eq:v:flat:Gevrey}\\
 \int_0^t \abs{v^\flat\vert_{y=0}(s)}_{\gamma,r,\tau(s)}^2 ds  &\les \frac{1}{\beta^{3}}  \int_0^t \abs{h(s)}_{\gamma,r+2\gamma-\frac{3}{2},\tau(s)}^2 ds    \label{eq:v:flat:y=0} \\
  \int_0^t \abs{v^\flat\vert_{y=1}(s)}_{\gamma,r,\tau(s)}^2 ds  &\les \frac{1}{\beta^{20}}  \int_0^t \abs{h(s)}_{\gamma,r+\gamma-10,\tau(s)}^2 ds 
   \label{eq:v:flat:y=1}
\end{align}
\end{subequations}
for all $t\in [0,T]$.
\end{lemma}
\begin{proof}[Proof of Lemma~\ref{lem:BL:1}]
In view of \eqref{eq:tau:def}, \eqref{eq:subindex:j}, and \eqref{eq:flat}, the function $\omega^\flat_j = M_j \pa_x^j \omega^\flat$ obeys equations
\begin{subequations} 
\begin{align}
\label{eq:omflatj1} (\pa_t + \beta (j+1) - \pa_y^2) \omega^{\flat}_j &= 0 \\ 
\label{eq:omflatj2} (\pa_y\omega^{\flat}_j + 2 \omega^{\flat}_j )\vert_{y=0} &= \pa_x h_j\vert_{y=0} = \frac{M_j}{M_{j+1}} h_{j+1}\\
\label{eq:omflatj3} \omega^{\flat}_j\vert_{t=0} &=0. 
\end{align}
\end{subequations}
For fixed $x \in \T$ we define $f_{j}(x,t) = \frac{M_j}{M_{j+1}} h_{j+1}(x,t)$ for $t\in[0,T]$, and $f_{j}(x,t) = 0$ for $t \in \R \setminus [0,T]$. Pointwise in $x$ and $y$ we take a Fourier transform in time and solve in $L^2(\R_t \times \T_x \times \R_{y}^+)$ the equation 
\begin{subequations}
\begin{align*}
(\pa_t + \beta (j+1) - \pa_y^2) \bar \omega^{\flat}_j &= 0 \\
(\pa_y \bar\omega^{\flat}_j + 2 \bar \omega^{\flat}_j )\vert_{y=0} &= f_j.
\end{align*}
\end{subequations}
The solution is obtained by taking the inverse Fourier transform in time (we let $\zeta$ denote the dual Fourier variable to $t$) of the function 
\begin{equation} \label{Fourier_omega} 
\hat{\bar\omega}^\flat_j(\zeta,x,y) = \frac{\hat f_j(\zeta,x)}{2 - \sqrt{\beta (j+1) + i \zeta}} e^{- y \sqrt{\beta (j+1)+i\zeta}}.
\end{equation}
We implicitly assume here that $\beta > 4$ so that for all $j \in \N$, for all $\zeta$ with $\Im \zeta \le 0$, 
\begin{equation} \label{lowerbound_Fourier}
  | 2 - \sqrt{\beta (j+1) + i \zeta} |  \ge |\sqrt{\beta (j+1) + i \zeta} | - 2 \ge     \sqrt{\beta(j+1) - \Im \zeta} -2 \ge \sqrt{\beta} - 2 > 0. 
\end{equation}
We will make a crucial use of 
\begin{lemma} \label{lemma_Fourier} 
 The following two properties hold 
\begin{itemize}
\item  $\bar \omega_j^\flat \equiv 0$ for $t <0$. 
\item  $\bar \omega_j^\flat \equiv \omega_j^\flat$ for $t\in [0,T]$. 
\end{itemize}
\end{lemma}
\noindent
The proof is postponed to Appendix \ref{appendA}. This lemma will allow us to use the explicit formula \eqref{Fourier_omega} to obtain estimates on $\omega_j^\flat$, starting with  \eqref{eq:omega:flat:Gevrey}-\eqref{eq:dy:omega:flat:y=1}. 

\mspace
Let us detail the derivation of \eqref{eq:omega:flat:Gevrey}.  A simple calculation based on \eqref{Fourier_omega} yields 
$$ \|  \hat{\bar\omega}^\flat_j \|^2_{L^2_{\zeta,x,y}} \le \frac{C}{(\beta (j+1))^{3/2}} \|  \hat{f}_j \|^2_{L^2_{\zeta,x}}    \, $$
for a constant $C$ independent of $j$ (and obviously from $T$, which is only involved in the definition of $f_j$). By Plancherel formula in time: 
\begin{equation} \label{estim_overomj}
     \|  \bar{\omega}^\flat_j \|^2_{L^2_{t,x,y}} \le \frac{C}{(\beta (j+1))^{3/2}} \|  f_j \|^2_{L^2_{t,x}} =   \frac{C}{\beta (j+1)^{3/2}}  \left(\frac{M_j}{M_{j+1}}\right)^2 \int_0^T  \|  h_{j+1}(s) \|_{L^2_x}^2  \, ds   
\end{equation}
This implies (by the second item of Lemma \ref{lemma_Fourier})
$$  \int_0^T \|  \omega^\flat_j(s) \|^2_{L^2_{x,y}}  ds \le  \frac{C'}{\beta^{3/2}} (j+1)^{2\gamma - \frac{3}{2}}   \int_0^T  \|  h_{j+1}(s) \|_{L^2_x}^2  \, ds  $$
Multiplying by $(j+1)^{2r}$ and summing over $j$, we obtain the inequality \eqref{eq:omega:flat:Gevrey} {\em in the special case $t=T$}.  For the general case $t \in (0,T)$, the idea is to slightly modify $\overline{\omega}_j^\flat$.  Namely, instead of extending $\frac{M_j}{M_{j+1}} h_{j+1}$ by zero outside $(0,T)$, and then solving the heat equation with the extension $f_j$ as a boundary data, we extend $\frac{M_j}{M_{j+1}} h_{j+1}\vert_{(0,t)}$ by zero outside $(0,t)$. We then solve the heat equation  with this modified boundary data $f_j^t$, which is zero outside $(0,t)$, resulting in  a new $\overline{\omega}_j^{\flat,t}$. Obviously,  Lemma \ref{lemma_Fourier} and the previous calculation remain true with $T$ replaced by $t$, $\overline{\omega}_j^\flat$ replaced by $\overline{\omega}_j^{\flat,t}$. 
This yields \eqref{eq:omega:flat:Gevrey}. Inequalities  \eqref{eq:y:omega:flat:Gevrey} to   \eqref{eq:y:u:flat:Gevrey} follow very similar arguments, that we skip for brevity. 

\mspace
In the case of \eqref{eq:v:flat:Gevrey}, we need to take into account one more $x$-derivative. A simple calculation yields (with obvious notations):
\begin{equation*} 
\|  \hat{\bar{v}}^\flat_j \|^2_{L^2_{\zeta,x,y}} \le \frac{C}{(\beta (j+1))^{7/2}} \|  \pa_x \hat{f}_j \|^2_{L^2_{\zeta,x}}    \, 
\end{equation*}
The extra factor of $(\beta (j+1))^2$ at the denominator compared to \eqref{estim_overomj} comes from taking two antiderivatives in $y$, while $\hat{f_j}$ is replaced by  $\pa_x \hat{f}_j$ due to the extra $x$-derivative in \eqref{def:vflat}. It follows that 
$$  \int_0^T \|  v^\flat_j(s) \|^2_{L^2_{x,y}}  ds \le  \frac{C}{\beta^{7/2}} (j+1)^{2\gamma - \frac{7}{2}}   \int_0^T  \|  \pa_x h_{j+1}(s) \|_{L^2_x}^2  \, ds  $$
and using that $|\pa_x h_{j+1}| \lesssim   \frac{M_{j+1}}{M_{j+2}} |h_{j+2}| \lesssim (j+2)^\gamma |h_{j+2}|$, we get 
$$   \int_0^T \|  v^\flat_j(s) \|^2_{L^2_{x,y}}  ds \le  \frac{C}{\beta^{7/2}} (j+1)^{4\gamma - \frac{7}{2}}   \int_0^T  \|  h_{j+2}(s) \|_{L^2_x}^2  \, ds.  $$
Multiplying by $(j+1)^{2r}$ and summing over $j$ yields \eqref{eq:v:flat:Gevrey} for $t=T$, while the case of an arbitrary time $t$ is treated with  the modification explained above. 
The pointwise estimate \eqref{eq:v:flat:y=0}, taken at $y=0$, follows from the inequality 
$$ \|  \hat{\bar{v}}^\flat_j\vert_{y=0} \|^2_{L^2_{\zeta,x}} \le  \frac{C}{(\beta (j+1))^{3}}  \|  \pa_x \hat{f}_j \|^2_{L^2_{\zeta,x}}. $$
The pointwise estimates \eqref{eq:dy:omega:flat:y=1}, \eqref{eq:u:flat:y=12}, and \eqref{eq:v:flat:y=1}, taken at $y=1$ or $y=1/2$ are much better: all boundary layer terms taken at $y=1$ contain an exponential factor $e^{-\sqrt{\beta(j+1) + i \xi}}$ which allows to gain an arbitrary number of powers of $\beta j$ (which explains the arbitrary factor $\frac{1}{\beta^{20}}$ and the index $r-\gamma-10$). 
\end{proof}

\begin{lemma}
\label{lem:BL:2}
Let  $r \in \R$, $\beta \ge 1$ and   $T>0$ such that $\tau(t) \geq \tau_1$ for $t \in [0,T]$. We have
\begin{subequations}
\begin{align}
\sup_{[0,t]} \norm{\omega^\flat(s)}_{\gamma,r,\tau(s)}^2  &\les \frac{1}{\beta^{1/2}}  \int_0^t \abs{h(s)}_{\gamma,r+\gamma-\frac{1}{4},\tau(s)}^2 ds 
\label{eq:lem:BL:2:a}
\end{align}
for all $t \in [0,T]$. 
\end{subequations}
\end{lemma}
\begin{proof}[Proof of Lemma~\ref{lem:BL:2}]
In order to establish the estimate \eqref{eq:lem:BL:2:a},  we rely on the explicit formula \eqref{Fourier_omega}, which gives an $L^1$ control of the Fourier transform: 
\begin{align*}
 \| \hat{\bar{\omega}}_j^\flat \|_{L^1_\zeta(L^2_x,y)} & \lesssim \int_\R \frac{1}{|\sqrt{\beta(j+1) + i \zeta} - 2|} \left( \int_{\R_+} \int_\T  \left| e^{-2y\sqrt{\beta(j+1) + i \zeta}}\right| \, |\hat{f}_j(\zeta,x)|^2 dx dy \right)^{1/2} d\zeta \\
 & \lesssim  \int_\R \frac{1}{|\sqrt{\beta(j+1) + i \zeta}|^{3/4}}  \left(  \int_\T  |\hat{f}_j(\zeta,x)|^2 dx\right)^{1/2}  d\zeta \\
 & \lesssim  \left( \int_\R \frac{1}{|\sqrt{\beta(j+1) + i \zeta}|^{3/2}} d\zeta \right)^{1/2}   \left( \int_\R \int_\T    |\hat{f}_j(\zeta,x)|^2 dx d\zeta \right)^{1/2} \\
 & \lesssim \frac{1}{(\beta(j+1))^{1/4}} \, \left( \int_\R \int_\T    |\hat{f}_j(\zeta,x)|^2 dx d\zeta \right)^{1/2}.
  \end{align*}
This implies that 
$$ \sup_{t \in \R} \| \overline{\omega}_j^\flat(t) \|_{L^2_{x,y}} \lesssim  \frac{1}{(\beta(j+1))^{1/4}} \,  \left( \int_\R \int_\T |f_{j+1}(t,x)|^2 dt \right)^{1/2}  $$
Restricting the left-hand side to the supremum over $(0,T)$, we get 
$$ \sup_{t \in (0,T)} \| \omega_j^\flat(t) \|^2_{L^2_{x,y}} \lesssim  \frac{1}{(\beta(j+1))^{-2\gamma+1/2}} \,  \int_0^T \int_\T \|h_{j+1}(t,x)|^2 dt.  $$
Multiplying by $(j+1)^{2r}$ and summing over $j$, we get \eqref{eq:lem:BL:2:a} for $t=T$. The general case of $t \in (0,T)$ is treated as in the proof of Lemma \ref{lem:BL:1}. 
\end{proof}

\subsection{The interior vorticity controls the boundary layer lift}
So far, we have only focused on the lower boundary layer lift, which is very small near $y=0$.  We introduce the notation
\begin{subequations}
\label{eq:bl:def}
\begin{align}
\ombl(x,y,t) &=   {\omega^\flat(x,y,t)  - \omega^\flat(x,1-y,t)}\\
\ubl(x,y,t) &=   {u^\flat(x,y,t) + u^\flat(x,1-y,t) }\\
\vbl(x,y,t) &= - \int_0^y \pa_x \ubl(x,z,t) dz
\end{align}
\end{subequations}
to denote the cumulative boundary layer profile, and
\begin{subequations}
\label{eq:in:def}
\begin{align}
\omi(x,y,t) &= \omega(x,y,t) - \ombl(x,y,t) \\
\uin(x,y,t) &= u(x,y,t) - \ubl(x,y,t) \\
\vin(x,y,t) &= v(x,y,t) - \vbl(x,y,t)
\end{align}
\end{subequations}
to denote the interior vorticity, horizontal velocity component, and vertical velocity component. In view of \eqref{eq:subindex:j}, \eqref{eq:bl:def} and \eqref{eq:in:def} also define the objects $\ombl_j, \ubl_j, \vbl_j$ in terms of the function $h$, and $\omi_j,\uin_j,\vin_j$ in terms of $h$ and $\omega$.

\begin{lemma}
\label{lem:h:omega:in}
Let  $\gamma \in [1, 5/4]$, ${r > 2\gamma+2}$, $M > 0$.  
Assume $\omega = \pa_y u$ is such that 
\begin{align}
\sup_{[0,T]}\norm{\omega(t)}_{\gamma,{\frac{r}{4}},\tau(t)} \leq M
\label{eq:h:omega:in:assume}
\end{align}
and define
\begin{align*}
 h(x,t) =  -  \int_0^1 (u(x,y,t))^2 \, dy + \int_\T \int_0^1 (u(x,y,t))^2 \, dy dx.
\end{align*}
With $h$ as above, let $\omega^\flat$   be defined via \eqref{eq:flat}, and let $\omi$ be as defined in \eqref{eq:in:def}.
Then there exists  $\beta_* = \beta_*(\tau_0,\tau_1,\gamma,r,M)$  such that: if $\beta \ge \beta_*$, if $T$ is such that $\tau(t) \ge \tau_1$ for $t \in [0,T]$, then 
\begin{align*}
 \int_0^t \abs{h(s)}_{\gamma,r,\tau(s)}^2 ds  &\les M^2 \int_0^t \norm{\omi(s)}_{\gamma,r,\tau(s)}^2 ds  
\end{align*}
for any $t \in [0,T]$.
\end{lemma}

\noindent
Note that with $h$ defined as above we have $\pa_x h = - \pa_x \int_0^1 u^2 \, dy$, so that the additional kinetic energy term in $h$ is not seen by $\omega^{bl}$. 
Combining Lemmas~\ref{lem:BL:1} and~\ref{lem:BL:2} and \ref{lem:h:omega:in}, we see that condition \eqref{eq:h:omega:in:assume} implies a sharp control of the Gevrey norm of the boundary layer profiles $\omega^{bl}$, $u^{bl}$, and $v^{bl}$, solely in terms of the Gevrey norm of the interior vorticity $\omega^{in}$ and of the constants $M$ and $\beta$. 
\begin{proof}[Proof of Lemma~\ref{lem:h:omega:in}]
For $j=0$ we have $h_0 = M_0 h = \tau h$, and since $\int_{\T} h(x,t) \, dx = 0$, we may apply the Poincar\'e inequality in the $x$ variable:
\begin{equation} \label{eq:h0:estimate} 
\norm{h_0}_{L^2_x} \les  \norm{\pa_x h_0}_{L^2_x} \les  \| h_1 \|_{L^2_x}. 
\end{equation}
Hence, it is enough to estimate $h_j$ for $j \ge 1$. By the Leibniz rule we have
\begin{align}
-  h_j(x,t) = \sum_{\ell=0}^j {j\choose \ell}  \frac{M_j}{M_{j-\ell} M_{\ell}} \int_0^1 u_\ell(x,y,t) u_{j-\ell}(x,y,t) dy.
\label{eq:hj:Leibniz}
\end{align}
We can without loss of generality estimate only the half-sum  $\sum_{0 \leq \ell \leq j/2}$, as the other half-sum can be put in the same form through the change of index $\ell' = j - \ell$. 

\mspace
First let us treat the case $\ell \geq 1$.
The compatibility condition \eqref{eq:u:compatibility} yields $\int_0^1 u_\ell(x,y) dy = 0$, 
which directly implies that 
\begin{align*}
\int_0^1 u_\ell(x,y) \uin_{j-\ell}(x,y) dy = \int_0^1 u_\ell(x,y) \left( \uin_{j-\ell}(x,y) - \int_0^1 \uin_{j-\ell}(x,z) dz \right) dy.
\end{align*}
Using the $1D$ Gagliardo-Nirenberg inequality, the $1D$ Hardy inequality, the $1D$ Poincar\'e inequality, and the fact that $u_\ell\vert_{y=0} = u_\ell\vert_{y=1} = 0$, we have that for $\ell\geq 1$:
\begin{align*}
&\norm{\int_0^1 u_\ell(x,y)  u_{j-\ell} (x,y) dy}_{L^2_x} \notag\\
&\quad \leq \norm{\int_0^1 u_\ell(x,y)  u_{j-\ell}^{in} (x,y) dy}_{L^2_x} + \norm{\int_0^1 u_\ell(x,y)  u_{j-\ell}^{bl} (x,y) dy}_{L^2_x} \notag\\
&\quad\leq \norm{u_\ell}_{L^\infty_x L^2_y} \norm{u_{j-\ell}^{in} - \int_0^1 \uin_{j-\ell} dz}_{L^2_{x,y}} + \norm{\frac{u_\ell}{y (1-y)}}_{L^\infty_x L^2_y} \norm{y(1-y) u_{j-\ell}^{bl}}_{L^2_{x,y}} \notag\\
&\quad\les \norm{u_\ell}_{L^2_{x,y}}^{1/2}  \norm{\pa_x u_\ell}_{L^2_{x,y}}^{1/2}   \norm{\omi_{j-\ell}}_{L^2_{x,y}}  + \norm{\omega_\ell}_{L^2_{x,y}}^{1/2}   \norm{\pa_x \omega_\ell}_{L^2_{x,y}}^{1/2}  \norm{y(1-y) u_{j-\ell}^{bl}}_{L^2_{x,y}}\notag\\
&\quad\les \frac{M_{\ell}^{1/2}}{M_{\ell+1}^{1/2}} \norm{\omega_\ell}_{L^2_{x,y}}^{1/2}   \norm{\omega_{\ell+1}}_{L^2_{x,y}}^{1/2} \left( \norm{\omega_{j-\ell}^{in}}_{L^2_{x,y}}  + \norm{y(1-y) u_{j-\ell}^{bl}}_{L^2_{x,y}} \right).
\end{align*}
For $\ell=0$, we estimate the $L^2_x$ norm of $\int_0^1 u_0 \, u_j^{bl} dy$ precisely as in the case $\ell\geq 1$. For the interior piece, since $j\geq 1$ we may use \eqref{eq:u:compatibility} and the Poincar\'e inequality in $y$ to estimate  
\begin{align*}
\norm{\int_0^1 u_0(x,y) u_j^{in}(x,y) dy }_{L^2_x} 
&\les \norm{u_0}_{L^\infty_x L^2_y} \left( \norm{u_j^{in}(x,y) - \int_0^1 u_j^{in} (x,z) dz}_{L^2_{x,y}} + \norm{\int_0^1 u_j^{bl}(x,z) dz}_{L^2_{x,y}} \right)
\\
&\les M \left( \norm{\omega_j^{in}}_{L^2_{x,y}} +\norm{\int_0^1 u_j^{bl}(x,z) dz}_{L^2_{x}}  \right)
\end{align*}
since $\norm{u_0}_{L^\infty_x L^2_y} \les \norm{\omega_0}_{L^\infty_x L^2_y} \les \norm{\omega_0}_{L^2_{x,y}} + \norm{\omega_1}_{L^2_{x,y}} \les M$. At this point we note that 
\begin{align*}
\int_0^1 u_j^{bl}(x,y) dy = - \int_0^{1/2} y \omega_j^{bl} (x,y) dy + u_j^{bl}(x,1/2) + \int_{1/2}^1 (1-y) \omega_j^{bl}(x,y) dz
\end{align*}
so that 
\begin{align*}
\norm{\int_0^1 u_j^{bl}(x,y) dy}_{L^2_{x}} \les \norm{y \omega_j^{\flat}}_{L^2_{x,y}} + \norm{u_j^{\flat}(x,1/2)}_{L^2_x}.
\end{align*}

\mspace
Returning to \eqref{eq:hj:Leibniz}, 
and using that in this range of $\ell$, namely less than $j/2$, we have
\begin{align*}
{j\choose \ell}  \frac{M_j}{M_{j-\ell} M_{\ell}^{1/2} M_{\ell+1}^{1/2}} \les \frac{1}{\tau^{1/2}} {j\choose \ell}^{1-\gamma} \frac{1}{(\ell+1)^{r-\gamma/2}} \les \frac{1}{(\ell+1)^{r-\gamma/2}},
\end{align*}
for $j\geq 1$ we obtain
\begin{align} \label{estim_hj}
\norm{h_j}_{L^2_x}  
&\les \sum_{\ell=1}^{\lceil j/2 \rceil} {j\choose \ell}  \frac{M_j}{M_{j-\ell} M_{\ell}^{1/2} M_{\ell+1}^{1/2}}  \norm{\omega_\ell}_{L^2_{x,y}}^{1/2}   \norm{\omega_{\ell+1}}_{L^2_{x,y}}^{1/2}  \left( \norm{\omega_{j-\ell}^{in}}_{L^2_{x,y}}  + \norm{y(1-y) u_{j-\ell}^{bl}}_{L^2_{x,y}} \right)  \notag\\
&\qquad + M \left(\norm{\omega_j^{in}}_{L^{2}_{x,y}} + \norm{y u_j^\flat}_{L^2_{x,y}} + \norm{y \omega_j^\flat}_{L^2_{x,y}}  +\norm{u_j^{\flat}(x,1/2)}_{L^2_x} \right) 
\notag\\
&\les  \sum_{\ell=1}^{\lceil j/2 \rceil} \frac{{(l+1)^{-\frac{3r}{4}}}\norm{\omega_\ell}_{L^2_{x,y}}^{1/2}  \norm{\omega_{\ell+1}}_{L^2_{x,y}}^{1/2}  }{(\ell+1)^{{\frac{r}{4}-\frac{\gamma}{2}}}} \left( \norm{\omega_{j-\ell}^{in}}_{L^2_{x,y}}  + \norm{y u_{j-\ell}^{\flat}}_{L^2_{x,y}} \right)  \notag\\
&\qquad + M \left(\norm{\omega_j^{in}}_{L^{2}_{x,y}} + \norm{y u_j^\flat}_{L^2_{x,y}} + \norm{y \omega_j^\flat}_{L^2_{x,y}}  +\norm{u_j^{\flat}(x,1/2)}_{L^2_x} \right).
\end{align}
From \eqref{eq:h0:estimate} and \eqref{estim_hj}, using the discrete H\"older and Young inequalities, inequalities \eqref{eq:y:u:flat:Gevrey}, \eqref{eq:u:flat:y=12}, \eqref{eq:y:omega:flat:Gevrey} and assumption \eqref{eq:h:omega:in:assume} we obtain from the above that
\begin{align*}
& \int_0^t \abs{h(s)}_{\gamma,r,\tau(s)}^2 ds
= \int_0^t  \sum_{j\geq 0} \norm{h_j(s)}_{L^2_{x}}^2 ds 
\notag\\
&\les 
\sup_{[0,t]} \left(\sum_{j\geq 0} \frac{{(j+1)^{-\frac{3r}{4}}} (\norm{\omega_{j}}_{L^2_{x,y}} +  \norm{\omega_{j+1}}_{L^2_{x,y}})}{ (j+1)^{ {\frac r4-\frac{\gamma}{2}}}}  \right)^2
 \int_0^t \Bigl(\sum_{j\geq 0}  \norm{\omega_{j}^{in}}_{L^2_{x,y}}^2  + \sum_{j\geq 0}   
 \norm{y u_{j}^{\flat}}_{L^2_{x,y}}^2 \Bigr) ds \notag\\
  &\qquad + M^2 \int_0^t \left( \norm{\omega^{in}(s)}_{\gamma,r,\tau(s)}^2 + \norm{y u^\flat(s)}_{\gamma,r,\tau(s)}^2 +   \norm{y \omega^\flat(s)}_{\gamma,r,\tau(s)}^2  +\abs{u^{\flat}(s)\vert_{y=1/2}}_{\gamma,r,\tau(s)}^2 \right) ds
 \notag\\
 &\les M^2 \left(\int_0^t \norm{\omi(s)}_{\gamma,r,\tau(s)}^2 ds   + \int_0^t \norm{y u^{\flat}(s)}_{\gamma,r,\tau(s)}^2 +   \norm{y \omega^\flat}_{\gamma,r,\tau(s)}^2  +\abs{u^{\flat}\vert_{y=1/2}}_{\gamma,r,\tau(s)}^2 ds \right) \notag\\
  &\les M^2 \left(\int_0^t \norm{\omi(s)}_{\gamma,r,\tau(s)}^2 ds   + \frac{1}{\beta^{5/2}} \int_0^t \abs{h(s)}_{\gamma,r+\gamma-\frac{5}{4},\tau(s)}^2 ds \right).
\end{align*}
Here we have used that $ {r/4-\gamma/2} > 1/2$. The proof is completed  using that $M^2 \beta^{-5/2} \ll 1$, which follows once $\beta_*$ is taken sufficiently large, and the fact that $\gamma \leq 5/4$, which allows us to absorb the second term in the right side of the above into the left side.
\end{proof}

\section{Estimates involving $\omi$}
From the vorticity evolution \eqref{eq:HNS:vort}, and the definition of $\ombl$ \eqref{eq:bl:def} (which in particular obeys $\int_\T \ombl (x,y,t) dx = 0$ for any $y\geq 0$), we obtain that the equation obeyed by the interior vorticity is
\begin{subequations}
\label{eq:omega:in}
\begin{align}
 \pa_t \omi - \pa_y^2 \omi + u \pa_x \omi + v \pa_y \omi
 &= - u\pa_x \ombl - v \pa_y \ombl
 \\
 \pa_y \omi\vert_{y=0,1} &=  \tilde \omega^{in}\vert_{y=1} - \tilde \omega^{in}\vert_{y=0}  + 2 \omega^\flat\vert_{y=1} - \pa_y \omega^\flat \vert_{y=1}.
 \\
 \omega^{in}(0) &= \omega_0
  \end{align}
\end{subequations}
The initial condition for $\omi$ is obtained from the fact that $\omega^{bl}(0) = 0$, which holds in view of \eqref{eq:flat:c}.
The main a priori estimate for $\omi$ is provided by the following Proposition. 

\begin{proposition}
\label{prop:omega:in}
Let $M, \delta_0$, $\gamma \in [1,9/8]$ be given, and let $\beta_*$ be as in Lemma~\ref{lem:h:omega:in}. There exists $r_0 = r_0(\gamma)$ such that for all $r \ge r_0$,  one can find   $\beta_0 = \beta_0(M,\delta_0,\tau_0,\tau_1,r,\gamma) > \max(\beta_*,4)$ satisfying: if $\beta \geq \beta_0$ and $T \leq 1$ is small enough so that $\tau(t) \ge \tau_1$ for all $t\in [0,T]$, under the assumptions
\begin{align}
 \sup_{t\in[0,T]} \norm{\omega(t)}_{\gamma, {\frac{3r}{4}},\tau(t)} +  \sup_{t\in[0,T]} \norm{\pa_y \omega(t)}_{\gamma, {\frac{r}{2}},\tau(t)} \leq M \label{eq:omi:ass:1}
\end{align}
and 
\begin{align}
\delta_0 \leq \pa_y  \omega &\leq \frac{1}{\delta_0}, \label{eq:omi:ass:2}\\
 \sup_{t\in [0,T]} \norm{\pa_y^2  \omega(t)}_{L^\infty_x L^2_y} &\leq M, \label{eq:omi:ass:3}
\end{align} 
we have that
\begin{equation} \label{eq:prop:omega:in:1}
 \sup_{s\in[0,t]} \norm{\omi(s)}_{\gamma,r,\tau(s)}^2 +   \int_0^t \norm{\pa_y \omi(s)}_{\gamma,r,\tau(s)}^2 ds  +  \beta \int_0^t  \norm{\omi(s)}_{\gamma,r+\frac 12,\tau(s)}^2   ds   \leq  {\frac{1}{\delta_0^2}}  \norm{\omega(0)}_{\gamma,r,\tau_0}^2 
\end{equation}
holds for all $t\in[0,T]$. Moreover, as a consequence we obtain
\begin{equation}
 \sup_{s\in[0,t]} \norm{\omega(s)}_{\gamma,r-\gamma + \frac 34,\tau(s)}^2 +  \int_0^t \norm{\pa_y \omega(s)}_{\gamma,r-\gamma + \frac 34,\tau(s)}^2 ds  +  \beta \int_0^t \norm{\omega(s)}_{\gamma,r -\gamma +\frac 54 ,\tau(s)}^2 ds  \leq   {\frac{4}{\delta_0^2}}  \norm{\omega(0)}_{\gamma,r,\tau_0}^2
\label{eq:omega:Gevrey:0}
\end{equation}
for all $t \in [0,T]$.
\end{proposition}
\begin{proof}[Proof of Proposition~\ref{prop:omega:in}]
Using the convention \eqref{eq:subindex:j}, from \eqref{eq:omega:in} we obtain
\begin{subequations}
\label{eq:omega:in:j}
\begin{align}
(\pa_t + \beta (j+1) -\pa_y^2 ) \omi_j  &+ (u \pa_x + v \pa_y) \omi_j + \vin_j \pa_y \omega \notag\\
&= -  (u \pa_x +v \pa_y )\ombl_j  - \vbl_j \pa_y \omega - M_j [\pa_x^j, u\pa_x + v\pa_y] \omega + v_j \pa_y \omega \label{eq:omega:in:j:a}\\
 \pa_y \omi_j\vert_{y=0,1} 
 &= \tilde \omega^{in}_j\vert_{y=1} - \tilde \omega^{in}_j\vert_{y=0}  +  2 \omega^\flat_j\vert_{y=1} - \pa_y \omega^\flat_j\vert_{y=1}.
 \label{eq:omega:in:j:b}
\end{align}
\end{subequations}
Note that as soon as $j\geq 1$, we may replace $\tilde \omega^{in}_j \vert_{y=0,1} = \omega^{in}_j \vert_{y=0,1}$ in \eqref{eq:omega:in:j:b}.
We perform a ``hydrostatic energy estimate'' on \eqref{eq:omega:in:j}, which is permissible in view of  \eqref{eq:omi:ass:2}. That is, we multiply \eqref{eq:omega:in:j:a} with $\omi_j/\pa_y \omega$ and integrate over $\Omega = \T \times [0,1]$.  We  notably  use  the ``hydrostatic trick'', which in this case gives
\begin{align*}
\int_\Omega \vin_j \omi_j dx dy &= - \int_\Omega \left( \int_0^y \pa_x \uin_j\right) \pa_y \uin_j dx dy 
\notag\\
&= \int_\Omega \pa_x \uin_j \uin_j dx dy - \int_\T  \left( \int_0^1 \pa_x \uin_j\right) \uin_j\vert_{y=1} dx 
\notag\\
&=   {-}\int_\T  \left( \int_0^1 \pa_x \ubl_j(x,y) dy \right) \ubl_j(x,1) dx.
\end{align*}
taking into account that $\int_0^1 \pa_x u_j(x,y) dy =0$ and that   $u_j\vert_{y=1} =0$.
Thus, we obtain
\begin{align}
&\frac{1}{2}\frac{d}{dt} \norm{\frac{\omi_j}{\sqrt{\pa_y \omega}}}_{L^2}^2 + \beta (j+1) \norm{\frac{\omi_j}{\sqrt{\pa_y \omega}}}_{L^2}^2 + \norm{\frac{\pa_y \omi_j}{\sqrt{\pa_y \omega}}}_{L^2}^2 
\notag\\
& = \int_\T \left( \frac{\pa_y \omi_j \omi_j}{\pa_y \omega}\Big\vert_{y={1}} -  \frac{\pa_y \omi_j \omi_j}{\pa_y \omega}\Big\vert_{y={0}} \right) dx + \int_\T  \left( \int_0^1 \pa_x \ubl_j(x,y) dy \right) \ubl_j(x,1) dx 
\notag\\
&\quad + \int_\Omega \frac{\pa_y \omi_j \omi_j}{\pa_y \omega} \frac{\pa_y^2 \omega}{\pa_y \omega} dxdy
- \frac 12 \int_\Omega \frac{(\omi_j)^2}{\pa_y \omega} \frac{(u \pa_x + v \pa_y) \pa_y \omega}{\pa_y \omega} dxdy
\notag\\
& \quad - \int_\Omega  u \pa_x  \ombl_j \frac{\omi_j}{\pa_y \omega} dx dy  - \int_\Omega v  \pa_y \ombl_j \frac{\omi_j}{\pa_y \omega} dx dy - \int_\Omega \vbl_j \omi_j dx dy \notag\\
&\quad - \sum_{k=1}^{j} \frac{M_j}{M_k M_{j-k+1}} {j\choose k} \int_\Omega u_k  \omega_{j-k+1} \frac{\omi_j}{\pa_y \omega} dx dy   - \sum_{k=1}^{j-1} \frac{M_j}{M_k M_{j-k}} {j\choose k} \int_\Omega v_k  \pa_y \omega_{j-k} \frac{\omi_j}{\pa_y \omega} dx dy \notag\\
&=: T_{1j} + T_{2j} + T_{3j} - T_{4j} - T_{5j} - T_{6j} - T_{7j} - T_{8j} - T_{9j}.
\label{eq:omega:in:j:1}
\end{align}

\mspace
Summing over $j$, and integrating on $[0,t)$, with $t\leq T$, we obtain that 
\begin{align}
&\norm{\omi(t)}_{\gamma,r,\tau(t)}^2 + 2 \beta \int_0^t  \norm{\omi}_{\gamma,r+1/2,\tau}^2   +  \int_0^t \norm{\pa_y \omi}_{\gamma,r,\tau}^2 \notag\\
&\le   {\frac{1}{\delta_0^2}} \norm{\omi_0}_{\gamma,r,\tau_0}^2 + \frac{1}{\delta_0} \int_0^t \sum_{j\geq 0} \left( |T_{1j}| - \frac{1}{2} \norm{\frac{\pa_y \omi_j}{\sqrt{\pa_y \omega}}}_{L^2}^2 \right) + |T_{2j}| + \left( |T_{3j}| + |T_{4j}| -   \frac{1}{2} \norm{\frac{\pa_y \omi_j}{\sqrt{\pa_y \omega}}}_{L^2}^2 \right) ds
\notag\\
&\qquad + \frac{1}{\delta_0} \int_0^t  \sum_{j\geq 0}   |T_{5j}| + |T_{6j}| + |T_{7j}| +  |T_{8j}| +  |T_{9j}| ds.
\label{eq:omega:in:j:2}
\end{align}
The rest of the proof is dedicated to estimating the nine terms on the right side of \eqref{eq:omega:in:j:2}.

\mspace
{\bf The $T_{1j}$ bound.} From \eqref{eq:HNS:vort:BC} and \eqref{eq:omega:in:j:b} we obtain that 
\begin{align*}
T_{1j} 
&=  \int_\T   \frac{\pa_y \omi_j\vert_{y=0,1} (\omi_j\vert_{y= {1}} - \omi_j\vert_{y= {0}})}{\pa_y \omega\vert_{y=0,1}}   dx 
\notag\\
&=  \int_\T   \frac{(\tilde\omega^{in}_j\vert_{y= {1}} - \tilde\omega^{in}_j\vert_{y={0}}) (\omi_j\vert_{y= {1}} - \omi_j\vert_{y=0})}{\pa_y \omega\vert_{y=0,1}}   dx 
+ \int_\T   \frac{ (2\omega^\flat_j\vert_{y=1} - \pa_y \omega_j^\flat\vert_{y=1})(\omi_j\vert_{y= {1}} - \omi_j\vert_{y= {0}})}{\pa_y \omega\vert_{y=0,1}}   dx
\notag\\
&=T_{11j} + T_{12j}.
\end{align*}
From the Gagliardo-Nirenberg inequality $\norm{f}_{L^\infty(0,1)} \leq \norm{f}_{L^2(0,1)} + 2 \norm{f}_{L^2(0,1)}^{1/2}   \norm{\pa_y f}_{L^2(0,1)}^{1/2}$, we have
\begin{align*}
\left| T_{11j} \right|
&\les \frac{1}{\delta_0} \left( \norm{\omi_j}_{L^2_{x,y}}^2 + \norm{\omi_j}_{L^2_{x,y}} \norm{\pa_y \omi_j}_{L^2_{x,y}} \right).
\end{align*}
Using Cauchy-Schwartz, we similarly obtain
\begin{align*}
|T_{12j}| \les |T_{11j}| + \frac{1}{\delta_0}  \left( \norm{\omega^\flat_j\vert_{y=1}}_{L^2_x}^2 + \norm{\pa_y \omega_j^\flat\vert_{y=1}}_{L^2_x}^2\right).
\end{align*}
Summing up the above two estimates, and summing over $j\geq 0$ we obtain that 
\begin{align*}
\sum_{j\geq 0} \left(|T_{1j}| - \frac{1}{2} \norm{\frac{\pa_y \omi_j}{\sqrt{\pa_y \omega}}}_{L^2}^2\right)
\les \frac{1}{\delta_0^2}\norm{ \omi }_{\gamma,r,\tau}^2 + \frac{1}{\delta_0}  \left( \abs{\omega^\flat_j\vert_{y=1}}_{\gamma,r,\tau}^2 + \abs{\pa_y \omega_j^\flat\vert_{y=1}}_{\gamma,r,\tau}^2\right).
\end{align*}
Using \eqref{eq:omega:flat:y=1}--\eqref{eq:dy:omega:flat:y=1}, and combining the resulting bound with Lemma~\ref{lem:h:omega:in} (which may be used due to assumption \eqref{eq:omi:ass:1}), we arrive at
\begin{align}
\int_0^t \sum_{j\geq 0} \left(|T_{1j}| - \frac{1}{2} \norm{\frac{\pa_y \omi_j}{\sqrt{\pa_y \omega}}}_{L^2}^2\right) 
&\les \frac{1}{\delta_0^2} \int_0^t \norm{ \omi }_{\gamma,r,\tau}^2 + \frac{1}{\delta_0 \beta^{20}} \int_0^t \abs{h}_{\gamma,r+\gamma-10,\tau}^2
\notag\\
&\les \frac{1}{\delta_0^2} \int_0^t \norm{ \omi }_{\gamma,r,\tau}^2 
\label{eq:T1j}
\end{align}
where we have used that $\delta_0 M^2 \leq \beta^{20}$.

\mspace
{\bf The $T_{2j}$ bound.} From \eqref{eq:bl:def} we obtain that
\begin{align*}
T_{2j} 
&= 2 \int_\T  \left( \int_0^1 \pa_x u^\flat_j(x,y) dy \right) \left(u^\flat_j(x,0) + u^\flat_j(x,1) \right) dx
\notag\\
&=  2 \int_\T  \left( v^\flat_j(x,0) - v^\flat_j(x,1) \right) \left(u^\flat_j(x,0) + u^\flat_j(x,1) \right) dx
\end{align*}
and thus, also appealing to Gagliardo-Nirenberg, we obtain
\begin{align*}
|T_{2j}| 
&\leq 2 \left( \norm{v^\flat_j\vert_{y=0}}_{L^2_x} + \norm{v^\flat_j \vert_{y=1}}_{L^2_x} \right) \left(\norm{u^\flat_j\vert_{y=0}}_{L^2_x} + \norm{u^\flat_j\vert_{y=1}}_{L^2_x} \right)
\notag\\
&\les \frac{\norm{v^\flat_j\vert_{y=0}}_{L^2_x} + \norm{v^\flat_j \vert_{y=1}}_{L^2_x}}{(j+1)^{\frac{3}{2}-\gamma}}
\left((j+1)^{\frac{3}{2}-\gamma} \norm{u_j^\flat}_{L^2_{x,y}} + (j+1)^{\frac{7}{8}-\frac{\gamma}{2}} \norm{u_j^\flat}_{L^2_{x,y}}^{1/2} (j+1)^{\frac{5}{8}-\frac{\gamma}{2}} \norm{\omega_j^\flat}_{L^2_{x,y}}^{1/2} \right),
\end{align*}
and summing over $j$ we arrive at
\begin{align*}
\sum_{j\geq 0} |T_{2j}| \les \left( \abs{v^\flat\vert_{y=0}}_{\gamma,r+\gamma-\frac{3}{2},\tau} + \abs{v^\flat_j \vert_{y=1}}_{\gamma,r+\gamma-\frac{3}{2},\tau} \right) 
\left( \norm{u^\flat}_{\gamma,r+\frac{3}{2}-\gamma,\tau} + \norm{u^\flat}_{\gamma,r+\frac{7}{4}-\gamma,\tau}^{1/2} \norm{\omega^\flat}_{\gamma,r+\frac{5}{4}-\gamma,\tau}^{1/2} \right)
\end{align*}
Upon integrating on $[0,t)$, the above terms are bounded using \eqref{eq:omega:flat:Gevrey}, \eqref{eq:u:flat:Gevrey}, \eqref{eq:v:flat:y=0}, and \eqref{eq:v:flat:y=1}, after which Lemma~\ref{lem:h:omega:in} is used  to yield
\begin{align*}
\int_0^t \sum_{j\geq 0} |T_{2j}| 
&\les \frac{1}{\beta^{5/2}}  \left( \int_0^t \abs{h}^2_{\gamma,r+3\gamma-3,\tau} \right)^{1/2} 
\left(  \left(\int_0^t \abs{h}^2_{\gamma,r+\frac 14,\tau} \right)^{1/2} + \left(\int_0^t \abs{h}^2_{\gamma,r+\frac 12,\tau} \right)^{1/2} \right)
\notag\\
&\les \frac{M^2}{\beta^{5/2}} 
\left( \int_0^t \norm{\omi}^2_{\gamma,r+3\gamma-3,\tau} \right)^{1/2} \left(\int_0^t \norm{\omi}^2_{\gamma,r+\frac 12,\tau} \right)^{1/2} 
\end{align*}
For the last inequality, we have applied  Lemma~\ref{lem:h:omega:in} to both factors at the right-hand side, which is legitimate under the assumptions  
$$ 
r + \min\{3\gamma-3,\frac12\} \ge  2 \gamma + 2, \quad \sup_{[0,T]}\norm{\omega(t)}_{\gamma,\frac{1}{4}\left(r+\max\{3\gamma-3,\frac12\}\right),\tau(t)} \leq M. 
$$
Both 	assumptions are satisfied for $r > r(\gamma)$ large enough, the second one being deduced from \eqref{eq:omi:ass:1}. Thus we have proven
\begin{align}
\int_0^t \sum_{j\geq 0} |T_{2j}|  \les \frac{M^2}{\beta^{5/2}} \int_0^t \norm{\omi}_{\gamma, r+\frac 12,\tau}^2.
\label{eq:T2j}
\end{align}

\mspace
{\bf The $T_{3j}$ and $T_{4j}$ bounds.} 
These are the only terms for which assumption \eqref{eq:omi:ass:3} is used.
In view of \eqref{eq:omi:ass:2}--\eqref{eq:omi:ass:3} and the Gagliardo-Nirenberg inequality in $y$,  we immediately obtain
\begin{align*}
&\sum_{j\geq 0} \left( |T_{3j}| - \frac 14 \norm{\frac{\pa_y \omi_j}{\sqrt{\pa_y \omega}}}_{L^2}^2 \right) 
\notag\\
&\les \sum_{j\geq 0} \left(\frac{M}{\delta_0^{3/2}} \norm{\frac{\pa_y \omi_j}{\sqrt{\pa_y \omega}}}_{L^2_x L^2_y} \norm{\omi_j}_{L^2_x L^\infty_y} - \frac 18 \norm{\frac{\pa_y \omi_j}{\sqrt{\pa_y \omega}}}_{L^2}^2 \right) 
\notag\\
&\les \sum_{j\geq 0} \left(\frac{M}{\delta_0^{3/2}} \norm{\frac{\pa_y \omi_j}{\sqrt{\pa_y \omega}}}_{L^2} \left( \norm{\omi_j}_{L^2} +\frac{1}{\delta_0^{1/4}} \norm{\omi_j}_{L^2}^{1/2} \norm{\frac{\pa_y \omi_j}{\sqrt{\pa_y \omega}}}_{L^2}^{1/2} \right)- \frac 18 \norm{\frac{\pa_y \omi_j}{\sqrt{\pa_y \omega}}}_{L^2}^2 \right) 
\notag\\
&\les    \frac{M^4}{\delta_0^7} \norm{ \omi}_{\gamma,r,\tau}^2
\end{align*}
and using \eqref{eq:omi:ass:1} combined with \eqref{eq:omi:ass:2}--\eqref{eq:omi:ass:3} we also obtain
\begin{align*}
&\sum_{j\geq 0} \left( |T_{4j}| - \frac 14 \norm{\frac{\pa_y \omi_j}{\sqrt{\pa_y \omega}}}_{L^2}^2 \right) 
\notag\\
& \les  \sum_{j\geq 0} \left( \frac{M^2}{\delta_0^2} \norm{\omi_j}_{L^2_x L^2_y} \norm{\omi_j}_{L^2_x L^\infty_y} - \frac 14 \norm{\frac{\pa_y \omi_j}{\sqrt{\pa_y \omega}}}_{L^2}^2 \right)
\notag\\
& \les  \sum_{j\geq 0} \left( \frac{M^2}{\delta_0^2} \norm{\omi_j}_{L^2} \left( \norm{\omi_j}_{L^2} + \frac{1}{\delta_0^{1/4}} \norm{\omi_j}_{L^2}^{1/2}  \norm{\frac{\pa_y \omi_j}{\sqrt{\pa_y \omega}}}_{L^2}^{1/2}\right) - \frac 14 \norm{\frac{\pa_y \omi_j}{\sqrt{\pa_y \omega}}}_{L^2}^2 \right)
\notag\\
&\les \frac{M^{8/3}}{\delta_0^3}  \norm{ \omi}_{\gamma,r,\tau}^2.
\end{align*}
Here we have also used the second term on the left side of \eqref{eq:omi:ass:1}, in order to estimate $\norm{\pa_x \pa_y \omega}_{L^\infty_x L^2_y}$. 
Thus, 
\begin{align}
\int_0^t \sum_{j\geq 0} \left( |T_{3j}| + |T_{4j}|  - \frac 14 \norm{\frac{\pa_y \omi_j}{\sqrt{\pa_y \omega}}}_{L^2}^2 \right) \les    \frac{M^4}{\delta_0^7} \int_0^t \norm{ \omi}_{\gamma,r,\tau}^2.
\label{eq:T34j}
\end{align}

\mspace
{\bf The $T_{5j}$ bound.} 
As it turns out, this term creates the most stringent assumption on $\gamma$, namely that $\gamma \leq 9/8$. Since $u\vert_{y=0,1} = 0$, using \eqref{eq:omi:ass:1} and \eqref{eq:omi:ass:3}, we have
\begin{align*}
|T_{5j}| 
&\leq \frac{1}{ \delta_0 }\norm{\frac{u}{y(1-y)}}_{L^\infty} \norm{y (1-y) \pa_x \ombl_j}_{L^2}\norm{ \omi_j }_{L^2}
\notag\\
&\les \frac{\norm{\omega}_{L^\infty}}{ \delta_0 } \frac{M_j}{M_{j+1} (j+1)^{1/2}}\norm{y \omega^\flat_{j+1}}_{L^2} (j+1)^{1/2} \norm{ \omi_j }_{L^2} 
\end{align*}
and thus, upon summing over $j$ and integrating on $[0,t]$ we arrive at
\begin{align*}
\int_0^t \sum_{j\geq 0} |T_{5j}| &\les \frac{M}{ \delta_0 } \left(\int_0^t \norm{y \omega^\flat }_{\gamma,r+\gamma-\frac 12,\tau}^2\right)^{1/2} \left(\int_0^t \norm{ \omi }_{\gamma,r+\frac 12,\tau}^2\right)^{1/2}.
\end{align*}
We now appeal to \eqref{eq:y:omega:flat:Gevrey} and to Lemma~\ref{lem:h:omega:in}, which is again legitimate for $r > r(\gamma)$ large enough. 
We obtain
\begin{align}
\int_0^t \sum_{j\geq 0} |T_{5j}| 
&\les \frac{M}{ \delta_0 \beta^{5/4} } \left(\int_0^t \abs{h}_{\gamma,r+2\gamma-\frac 74,\tau}^2\right)^{1/2} \left(\int_0^t \norm{ \omi }_{\gamma,r+\frac 12,\tau}^2\right)^{1/2}
\notag\\
&\les \frac{M^2}{ \delta_0 \beta^{5/4} } \left(\int_0^t \norm{\omi}_{\gamma,r+2\gamma-\frac 74,\tau}^2\right)^{1/2} \left(\int_0^t \norm{ \omi }_{\gamma,r+\frac 12,\tau}^2\right)^{1/2}
\notag\\
&\les \frac{M^2}{ \delta_0 \beta^{5/4} } \int_0^t \norm{ \omi }_{\gamma,r+\frac 12,\tau}^2.
\label{eq:T5j}
\end{align}
In the last inequality we have used that  $2\gamma - 7/4 \leq 1/2$, which holds since $\gamma \leq 9/8$.

\mspace
{\bf The $T_{6j}$ bound.} Similarly, using that $v\vert_{y=0,1} = 0$, we obtain
\begin{align*}
|T_{6j}| 
&\leq \frac{1}{ \delta_0 }\norm{\frac{v}{y(1-y)}}_{L^\infty} \norm{y (1-y) \pa_y \ombl_j}_{L^2}\norm{ \omi_j }_{L^2}
\notag\\
&\les \frac{\norm{\pa_x u}_{L^\infty}}{\delta_0} \norm{y \pa_y \omega^\flat_{j}}_{L^2} \norm{ \omi_j }_{L^2} \notag\\
&\les \frac{M}{\delta_0} \frac{\norm{y \pa_y\omega^\flat_{j}}_{L^2}}{(j+1)^{1/2}}  \left( (j+1)^{1/2} \norm{ \omi_j }_{L^2} \right),
\end{align*}
so that 
\begin{align*}
\int_0^t \sum_{j\geq 0} |T_{6j}| \les \frac{M}{\delta_0}\left( \int_0^t \norm{y \pa_y \omega^\flat}_{\gamma,r-\frac 12,\tau}^2 \right)^{1/2} \left(\int_0^t  \norm{\omi}_{\gamma,r+\frac 12,\tau}^2 \right)^{1/2}
\end{align*}
Using \eqref{eq:y:dy:omega:flat:Gevrey}, and then Lemma~\ref{lem:h:omega:in} (applicable for $r > r(\gamma)$ large enough, by \eqref{eq:omi:ass:1}), we obtain
\begin{align}
\int_0^t \sum_{j\geq 0} |T_{6j}| 
&\les \frac{M}{\delta_0 \beta^{3/4} }\left( \int_0^t \abs{h}_{\gamma,r+\gamma-\frac 74,\tau}^2 \right)^{1/2} \left(\int_0^t  \norm{\omi}_{\gamma,r+\frac 12,\tau}^2 \right)^{1/2}
\notag\\
&\les \frac{M^2}{\delta_0 \beta^{3/4} }\left( \int_0^t \norm{\omi}_{\gamma,r,\tau}^2 \right)^{1/2} \left(\int_0^t  \norm{\omi}_{\gamma,r+\frac 12,\tau}^2 \right)^{1/2}
\label{eq:T6j}
\end{align}
since $\gamma \leq 7/4$.

\mspace
{\bf The $T_{7j}$ bound.} For $T_{7j}$ we directly estimate
\begin{align*}
\sum_{j\geq 0} |T_{7j}| \leq \sum_{j\geq 0}  \frac{1}{ \delta_0 } (j+1)^{-1/2} \norm{v^\flat_j}_{L^2}  (j+1)^{1/2} \norm{  \omi_j }_{L^2} \les \frac{1}{\delta_0} \norm{v^\flat}_{\gamma,r-\frac 12,\tau} \norm{\omi}_{\gamma,r+\frac 12,\tau}.
\end{align*}
Integrating in time, appealing to \eqref{eq:v:flat:Gevrey}, and still using Lemma~\ref{lem:h:omega:in} we obtain
\begin{align}
\int_0^t \sum_{j\geq 0} |T_{7j}|
&\les \frac{1}{\delta_0 \beta^{7/4}} \left( \int_0^t \abs{h}_{\gamma,r+2\gamma-\frac 94,\tau}^2 \right)^{1/2} \left(\int_0^t  \norm{\omi}_{\gamma,r+\frac 12,\tau}^2 \right)^{1/2}
\notag\\
&\les \frac{M}{\delta_0 \beta^{7/4}} \left( \int_0^t \norm{\omi}_{\gamma,r+2\gamma-\frac 94,\tau}^2 \right)^{1/2} \left(\int_0^t  \norm{\omi}_{\gamma,r+\frac 12,\tau}^2 \right)^{1/2}
\notag\\
&\les \frac{M}{\delta_0 \beta^{7/4}} \int_0^t  \norm{\omi}_{\gamma,r+\frac 12,\tau}^2 
\label{eq:T7j}
\end{align}
as $2\gamma - 9/4 \leq 1/2$.

\mspace
{\bf The $T_{8j}$ bound.} We note that
\begin{align*}
\frac{M_j}{M_k M_{j-k+1}} {j \choose k} 
\les {j \choose k}^{1-\gamma} \frac{(j+1)^r}{(k+1)^r (j-k+1)^{r-\gamma}},
\end{align*}
and for $1 \leq k \leq [j/2]$ it is convenient to use  ${j \choose k} \geq (j-k+1)/k$. We obtain
\begin{align*}
|T_{8j}| 
&\les \sum_{k=1}^{[j/2]} \frac{j^{1/2} (j-k+1)^{1/2}}{(k+1)^{r-\gamma+1} } \left| \int_\Omega u_k  \omega_{j-k+1} \frac{\omi_j}{\pa_y \omega} \right|  + \sum_{k=[j/2]+1}^{j} \frac{1}{(j-k+1)^{r-\gamma}} \left| \int_\Omega u_k  \omega_{j-k+1} \frac{\omi_j}{\pa_y \omega}  \right|  \notag\\
&=: T_{8j,{\rm low}} + T_{8j, {\rm high}}.
\end{align*}
In order to estimate $T_{8j,{\rm low}}$, we split $\omega_{j-k+1} = \omega_{j-k+1}^{in} + \omega_{j-k+1}^{bl}$. First, using the Gagliardo-Nirenberg inequality on $\Omega$ and the Poincar\'e inequality in $x$ (since $k\geq 1$) we may bound
\begin{align}
\norm{\omega_k}_{L^\infty} 
&\les \norm{\omega_k}_{L^2} + \norm{\pa_x \omega_k}_{L^2} + (\norm{\omega_k}_{L^2}^{1/2} +  \norm{\pa_x \omega_k}_{L^2}^{1/2}) (\norm{\pa_y \omega_k}_{L^2}^{1/2} + \norm{\pa_{x} \pa_y \omega_k}_{L^2}^{1/2}) \notag\\
&\les \norm{\pa_x \omega_k}_{L^2} +   \norm{\pa_x \omega_k}_{L^2}^{1/2} \norm{\pa_{x} \pa_y \omega_k}_{L^2}^{1/2}
\notag\\
&\les k^\gamma \left( \norm{\omega_{k+1}}_{L^2} +  \norm{ \pa_y \omega_{k+1}}_{L^2} \right)
\label{eq:omega:k:L:infty}
\end{align} 
from which we conclude that
we estimate
\begin{align*}
\left| \int_\Omega u_k  \ombl_{j-k+1} \frac{\omi_j}{\pa_y \omega} dx dy \right|
&\les \frac{1}{ \delta_0 }\norm{\frac{u_k}{y(1-y)}}_{L^\infty} \norm{y (1-y) \omega^{bl}_{j-k+1}}_{L^2} \norm{ \omi_j  }_{L^2}
\notag\\
&\les \frac{k^\gamma}{ \delta_0 } \left( \norm{\omega_{k+1}}_{L^2} + \norm{ \pa_y \omega_{k+1}}_{L^2} \right) \norm{y \omega^\flat_{j-k+1}}_{L^2} \norm{ \omi_j }_{L^2}
\notag\\
&\les \frac{k^{\gamma+r/2}}{ \delta_0 } \frac{\norm{\omega_{k+1}}_{L^2} + \norm{ \pa_y \omega_{k+1}}_{L^2} }{k^{r/2}} \norm{y \omega^\flat_{j-k+1}}_{L^2} \norm{ \omi_j }_{L^2}.
\end{align*}
Similarly, 
\begin{align*}
\left| \int_\Omega u_k  \omi_{j-k+1} \frac{\omi_j}{\pa_y \omega} dx dy \right|
&\les \frac{k^{\gamma+r/2}}{ \delta_0 }  \frac{ \norm{\omega_{k+1}}_{L^2} +   \norm{ \pa_y \omega_{k+1}}_{L^2} }{k^{r/2}} \norm{ \omi_{j-k+1} }_{L^2}
\norm{ \omi_j }_{L^2}
\end{align*}
so that from the discrete Young and H\"older inequalities, we obtain 
\begin{align}
& \sum_{j\geq 0} T_{8j,{\rm low}} \notag \\
& \les  \frac{1}{\delta_0} \left( \sum_{j\neq 0} \frac{j^{\gamma+r/2}}{(j+1)^{r-\gamma+1}} \frac{ \norm{\omega_{j+1}}_{L^2} +   \norm{ \pa_y \omega_{j+1}}_{L^2} }{j^{r/2}} \right)  \,  \left( \norm{y\,  \omega^\flat}_{\gamma,r+\frac 12,\tau} + \norm{\omi}_{\gamma,r+\frac 12,\tau}\right)  \norm{\omi}_{\gamma,r+\frac 12,\tau}
\notag \\ 
&\les \frac{1}{\delta_0} \left( \norm{\omega}_{\gamma,\frac r2} + \norm{\pa_y \omega}_{\gamma, \frac r2} \right)  \left( \norm{y\,  \omega^\flat}_{\gamma,r+\frac 12,\tau} + \norm{\omi}_{\gamma,r+\frac 12,\tau}\right)  \norm{\omi}_{\gamma,r+\frac 12,\tau}
\notag\\
&\les\frac{M}{\delta_0}  \left( \norm{y\,  \omega^\flat}_{\gamma,r+\frac 12,\tau} + \norm{\omi}_{\gamma,r+\frac 12,\tau}\right) \norm{\omi}_{\gamma,r+\frac 12,\tau}.
\label{eq:T8j:low}
\end{align}
For the second inequality, we have assumed  that $r/2 - 2\gamma +1 > 1/2$ (so that $\frac{j^{\gamma+r/2}}{(j+1)^{r-\gamma+1}} $ is square summable), and for the third inequality we have appealed to \eqref{eq:omi:ass:1}.  

\mspace
In order to bound $T_{8j,{\rm high}}$, we use that $u_k \vert_{y=0,1} = 0$, and the 1D Poincar\'e inequality  to  obtain
\begin{align*}
\left| \int_\Omega u_k  \omega_{j-k+1} \frac{\omi_j}{\pa_y \omega} dx dy \right|
&\les \frac{1}{\delta_0} \norm{u_k}_{L^2_x L^\infty_y} \norm{\omega_{j-k+1}}_{L^\infty_x  L^2_y} \norm{ \omi_j }_{L^2}
\notag\\
&\les \frac{ (j-k+1)^\gamma }{\delta_0}  \norm{\omega_k}_{L^2} \norm{\omega_{j-k+2}}_{L^2} \norm{\omi_j}_{L^2}
\notag\\
&\les \frac{ (j-k+1)^\gamma }{\delta_0} \frac{\norm{\omi_k}_{L^2} + \norm{\ombl_k}_{L^2}}{(k+1)^{1/2}} \norm{\omega_{j-k+2}}_{L^2} (j+1)^{1/2} \norm{\omi_j}_{L^2}.
\end{align*}
We again rely on discrete Young and H\"older inequalities,  assume that  $r >  \frac{8}{3} \gamma + \frac{2}{3}$ (so that 
$(j+1)^{2\gamma-3r/4}$ is square summable), and    use \eqref{eq:omi:ass:1} to arrive at
\begin{align}
 \sum_{j\geq 0} T_{8j,{\rm high}} & \les  \frac{1}{\delta_0}  \left( \sum_j (j+1)^{2\gamma - {3r/4}}   \frac{\norm{\omega_j}_{L^2}}{(j+1)^{{r/4}}} \right) \,  \norm{\omi}_{\gamma,r+\frac 12,\tau} \left( \norm{\omi}_{\gamma,r,\tau}  + \norm{\omega^\flat}_{\gamma,r- \frac 12,\tau}  \right)
 \notag \\
 & \les \frac{M}{\delta_0} \norm{\omi}_{\gamma,r+\frac 12,\tau} \left( \norm{\omi}_{\gamma,{r-\frac{1}{2}},\tau}  + \norm{\omega^\flat}_{\gamma,r- \frac 12,\tau}  \right).
\label{eq:T8j:high}
\end{align}

\mspace
Combining \eqref{eq:T8j:low}, \eqref{eq:T8j:high},  integrating in time, using \eqref{eq:omega:flat:Gevrey}, \eqref{eq:y:omega:flat:Gevrey}, and Lemma~\ref{lem:h:omega:in} (which is applicable by assumption \eqref{eq:omi:ass:1}),  we arrive at
\begin{align}
\int_0^t \sum_{j\geq 0} T_{8j} 
\les & \: 
\frac{M}{\delta_0}  \left(\left( \int_0^t  \norm{y\,  \omega^\flat}_{\gamma,r+\frac 12,\tau}^2 \right)^{1/2} + \left( \int_0^t  \norm{\omega^\flat}_{\gamma,r-\frac 12,\tau}^2 \right)^{1/2}  \right)   \left( \int_0^t \norm{\omi}_{\gamma,r+\frac 12,\tau}^2 \right)^{1/2}  \notag\\
 + & \:   \frac{M}{\delta_0} \int_0^t \norm{\omi}_{\gamma,r+\frac 12,\tau}^2 
\notag\\
\les & 
\frac{M}{\delta_0 \beta^{3/4}}  \left( \left( \int_0^t  \abs{h}_{\gamma,r + \gamma - \frac 34,\tau}^2 \right)^{1/2} + \left( \int_0^t  \abs{h}_{\gamma,r+\gamma-\frac 54,\tau}^2 \right)^{1/2} \right) \left(  \int_0^t \norm{\omi}_{\gamma,r+\frac 12,\tau}^2 \right)^{1/2}  \notag\\
+  & \: \frac{M}{\delta_0} \int_0^t \norm{\omi}_{\gamma,r+\frac 12,\tau}^2 
\notag\\
\les &\:  \frac{M^2}{\delta_0 \beta^{3/4}}  \left( \int_0^t  \norm{\omi}_{\gamma,r + \gamma - \frac 34,\tau}^2 \right)^{1/2}  \left( \int_0^t \norm{\omi}_{\gamma,r+\frac 12,\tau}^2 \right)^{1/2} \: + \: \frac{M}{\delta_0}  \int_0^t \norm{\omi}_{\gamma,r+\frac 12,\tau}^2 \notag \\
\les & \: \frac{M^2}{\delta_0} \int_0^t \norm{\omi}_{\gamma,r+\frac 12,\tau}^2
\label{eq:T8j}
\end{align}
since $\gamma \leq 5/4$.

\mspace
{\bf The $T_{9j}$ bound.} In order to estimate $T_{9j}$ we note that for $1\leq k \leq j-1$ we have
\begin{align*}
\frac{M_j}{M_k M_{j-k}} {j\choose k} \les {j\choose k}^{1-\gamma}   \frac{(j+1)^r}{(k+1)^r (j-k+1)^{r}} 
\les \left( \frac{j}{\min\{ k, j-k \}}\right)^{1-\gamma} \frac{1}{( \min\{k, j-k\})^r}
\end{align*}
and similarly to $T_{8j}$ we decompose 
\begin{align}
T_{9j} 
&\les \sum_{k=1}^{[j/2]} \frac{1}{k^{r}} \left| \int_\Omega v_k  \pa_y \omega_{j-k} \frac{\omi_j}{\pa_y \omega} \right| 
+ \sum_{k=[j/2]+1}^{j-1}  \frac{1}{(j-k)^{r-\gamma+1} j^{\gamma-1}} \left| \int_\Omega v_k  \pa_y \omega_{j-k} \frac{\omi_j}{\pa_y \omega} \right|
\notag\\
&=: T_{9j,{\rm low}} + T_{9j,{\rm high}}.
\end{align}
First we treat the case $k \leq j/2$. Using the Poincar\'e inequality in $y$ (which is allowed since $u_{k+1} \vert_{y=0,1} = 0$)  we obtain
\begin{align*}
\left| \int_\Omega v_k  \pa_y \omega_{j-k} \frac{\omi_j}{\pa_y \omega} dx dy\right|
&\les \frac{1}{\delta_0} \norm{\frac{v_k}{y(1-y)}}_{L^\infty} \norm{y(1-y)\pa_y \omega_{j-k}}_{L^2} \norm{\omi_j}_{L^2} \notag\\
&\les \frac{1}{\delta_0} \norm{\pa_x u_k}_{L^\infty} \left( \norm{\pa_y \omi_{j-k}}_{L^2} + \norm{y \pa_y \omega^\flat_{j-k}}_{L^2} \right) \norm{\omi_j}_{L^2} \notag\\
&\les \frac{k^{\gamma}}{\delta_0} \norm{\omega_{k+1}}_{L^\infty_x L^2_y} \left( \norm{\pa_y \omi_{j-k}}_{L^2} +  \norm{y \pa_y \omega^\flat_{j-k}}_{L^2} \right)  \norm{\omi_j}_{L^2}
\end{align*}
Furthermore, using the 1D Gagliardo-Nirenberg and Poincar\'e inequalities in $x$, for $1\leq k \leq [j/2]$ we arrive at
\begin{align*}
\left| \int_\Omega v_k  \pa_y \omega_{j-k} \frac{\omi_j}{\pa_y \omega} dx dy\right|  \les \frac{k^{2\gamma+r/4}}{\delta_0} \frac{\norm{\omega_{k+2}}_{L^2}}{k^{r/4}} \left( \norm{\pa_y \omi_{j-k}}_{L^2} +  \norm{y \pa_y \omega^\flat_{j-k}}_{L^2}   \right)   \norm{\omi_j}_{L^2}.
\end{align*}
Summing over $j$, assumng that $r > \frac{8}{3}\gamma +\frac{2}{3}$, and appealing to~\eqref{eq:omi:ass:1} we obtain
\begin{align}
\sum_{j\geq 0}  |T_{9j,{\rm low}}| 
&\les \frac{\norm{\omega}_{\gamma,\frac{3r}{4},\tau}}{\delta_0}   \left( \norm{\pa_y \omi}_{\gamma,r,\tau} + \norm{y \pa_y \omega^\flat}_{\gamma,r,\tau} \right) \norm{\omi}_{\gamma,r,\tau}
\notag\\
&\les \frac{M}{\delta_0}   \left( \norm{\pa_y \omi}_{\gamma,r,\tau} + \norm{y \pa_y \omega^\flat}_{\gamma,r ,\tau} \right) \norm{\omi}_{\gamma,r,\tau}
\label{eq:T9j:low}.
\end{align}

\mspace
For the case $k \geq j/2$, we first note that the compatibility condition \eqref{eq:u:compatibility}  allows us to write
\begin{align*}
\int_\T \int_0^1 u_{k+1}^2 dy dx =  \int_\T \int_0^1 u_{k+1} u_{k+1}^{bl}  dy dx +  \int_\T \int_0^1 u_{k+1} \left( u_{k+1}^{in}  - \int_0^1 u_{k+1}^{in} dz\right) dy dx.
\end{align*}
By Cauchy-Schwartz and the Poincar\'e inequality in $y$ (for zero mean functions) we  conclude
\begin{align*}
\norm{u_{k+1}}_{L^2}^2 \les \norm{u_{k+1}^{bl}}_{L^2}^2 + \norm{\omega_{k+1}^{in}}_{L^2}^2.
\end{align*}
Then we similarly estimate
\begin{align*}
&\left| \int_\Omega v_k  \pa_y \omega_{j-k} \frac{\omi_j}{\pa_y \omega} dx dy\right| 
\notag\\
&\quad \les \frac{1}{\delta_0} \norm{ v_k }_{L^2_x L^\infty_y} \norm{ \pa_y \omega_{j-k}}_{L^\infty_x L^2_y} \norm{\omi_j}_{L^2}
\notag\\
&\quad \les \frac{1}{\delta_0} \norm{\pa_x u_k}_{L^2}  \norm{\pa_x \pa_y \omega_{j-k}}_{L^2}  \norm{\omi_j}_{L^2}
\notag\\
&\quad \les  {\frac{(j-k)^{\gamma} j^{ \gamma-1}}{\delta_0}}   {k^{1/2}}  \norm{u_{k+1}}_{L^2}  \norm{\pa_y \omega_{j-k+1}}_{L^2} \left(j^{1/2} \norm{\omi_j}_{L^2}\right)
\notag\\
&\quad \les  {\frac{(j-k)^{\gamma+r/2} j^{ \gamma-1}}{\delta_0}} \left(  {k^{1/2}} \norm{\omi_{k+1}}_{L^2}  +  {k^{1/2}}  \norm{u^\flat_{k+1}}_{L^2}  \right) \frac{\norm{\pa_y \omega_{j-k+1}}_{L^2}}{(j-k)^{ {r/2}}} \left(j^{1/2} \norm{\omi_j}_{L^2}\right).
\end{align*}
Summing over $j$, noting that the powers of $j$ precisely cancel, we find for $r > r(\gamma)$ large enough:
\begin{align}
\sum_{j\geq 0}  |T_{9j,{\rm high}}|  
&\les \frac{\norm{\pa_y \omega}_{\gamma,\frac r2}}{\delta_0} \left( \norm{\omi}_{\gamma,r+\frac 12,\tau} + \norm{u^\flat}_{\gamma,r +\frac 12,\tau}\right) \norm{\omi}_{\gamma,r+\frac 12,\tau} 
\notag\\
&\les \frac{M}{\delta_0} \left( \norm{\omi}_{\gamma,r+\frac 12,\tau} + \norm{u^\flat}_{\gamma,r+\frac 12,\tau}\right) \norm{\omi}_{\gamma,r+\frac 12,\tau}.
\label{eq:T9j:high}
\end{align} 
Integrating in time the sum of \eqref{eq:T9j:low} and \eqref{eq:T9j:high},  appealing to \eqref{eq:omega:flat:Gevrey} and \eqref{eq:y:dy:omega:flat:Gevrey}, and using Lemma~\ref{lem:h:omega:in} (which is applicable for $r > r(\gamma)$ large enough, by assumption \eqref{eq:omi:ass:1}), we obtain
\begin{align}
\int_0^t \sum_{j\geq 0} |T_{9j}| -  {\frac{1}{2}} \int_0^t  \norm{\pa_y \omi}_{\gamma,r,\tau}^2
&\les \int_0^t \left( \norm{y \pa_y \omega^\flat}_{\gamma,r,\tau}^2  + \norm{u^\flat}_{\gamma,r+\frac 12,\tau}^2\right)  + \frac{M^2}{\delta_0^2} \int_0^t \norm{\omi}_{\gamma,r+\frac 12,\tau}^2
\notag\\
&\les \frac{1}{\beta^{3/2}} \int_0^t  \abs{h}_{\gamma,r+\gamma -\frac 34,\tau}^2  + \frac{M^2}{\delta_0^2} \int_0^t \norm{\omi}_{\gamma,r+\frac 12,\tau}^2
\notag\\
&\les \left( \frac{M^2}{\beta^{3/2}} +  \frac{M^2}{\delta_0^2} \right) \int_0^t \norm{\omi}_{\gamma,r+\frac 12,\tau}^2
\label{eq:T9j}
\end{align}
since $\gamma - 3/4 \leq \frac{1}{2}$.

\mspace
{\bf Conclusion of the proof.}
Inserting the bounds \eqref{eq:T1j}, \eqref{eq:T2j}, \eqref{eq:T34j}, \eqref{eq:T5j}, \eqref{eq:T6j}, \eqref{eq:T7j}, \eqref{eq:T8j}, and \eqref{eq:T9j} into estimate \eqref{eq:omega:in:j:2}, we obtain
\begin{align}
&\norm{\omi(t)}_{\gamma,r,\tau(t)}^2 +2  \beta \int_0^t   \norm{\omi}_{\gamma,r+1/2,\tau}^2  ds +  \int_0^t \norm{\pa_y \omi}_{\gamma,r,\tau}^2 ds -  \frac{1}{\delta_0^2} \norm{\omi_0}_{\gamma,r,\tau_0}^2 \notag\\
&\les    \left(\frac{1}{\delta_0^3} + \frac{M^4}{\delta_0^8}  + \frac{M}{\delta_0 \beta^{3/2}} \right) \int_0^t \norm{ \omi}_{\gamma,r,\tau}^2 ds 
\notag\\
&\qquad + \left( \frac{M^2}{\delta_0 \beta^{5/2}} + \frac{M^2}{\delta_0^2 \beta^{5/4}} + \frac{M^2}{\delta_0 \beta^{3/2}} + \frac{M}{\delta_0^2 \beta^{7/4}} + \frac{M^2}{\delta_0^2 \beta^{3/4}} + \frac{M^2}{\delta_0^3}\right) \int_0^t \norm{\omi}_{\gamma,r+\frac 12,\tau}^2   ds.
\label{eq:omega:in:j:3}
\end{align}
Note that $\norm{\omi}_{\gamma,r,\tau}^2 \leq  \norm{\omi}_{\gamma,r+\frac 12,\tau}^2$, so that we may combine the last two terms on the right side of \eqref{eq:omega:in:j:3}.
Choosing $\beta_0$ large enough, depending on $M\geq 1$, $\delta_0\leq 1$, and the implicit constant in \eqref{eq:omega:in:j:3}, for any $\beta \geq \beta_0$ we obtain
\begin{align*}
\norm{\omi(t)}_{\gamma,r,\tau(t)}^2 +  \beta  \int_0^t  \norm{\omi}_{\gamma,r+ \frac 12,\tau}^2   ds + \int_0^t \norm{\pa_y \omi}_{\gamma,r,\tau}^2 ds  \le \frac{1}{\delta_0^2}  \norm{\omi_0}_{\gamma,r,\tau_0}^2 .
\end{align*}
The estimate \eqref{eq:prop:omega:in:1} now follows directly from the above estimate.

\mspace
Finally, in order to prove \eqref{eq:omega:Gevrey:0}, we appeal to \eqref{eq:lem:BL:2:a}, Lemma~\ref{lem:h:omega:in}, and estimate \eqref{eq:prop:omega:in:1}, to obtain
\begin{align}
\sup_{[0,t]} \norm{\omega^\flat}_{\gamma,r-\gamma+\frac 34 ,\tau(s)}^2 
&\les \frac{1}{\beta^{1/2}} \int_0^t \abs{h(s)}_{\gamma,r + \frac 12,\tau(s)}^2 ds 
\notag\\
&\les \frac{M^2}{\beta^{1/2}} \int_0^t \norm{\omi(s)}_{\gamma,r+\frac 12,\tau(s)}^2 ds 
\leq \frac{1}{2\delta_0^2}  \norm{\omi(0)}_{\gamma,r,\tau_0}^2 
\label{eq:aux:1}
\end{align}
upon ensuring that $\beta$ is sufficiently large, depending on $M, \delta_0$. Moreover, from \eqref{eq:dy:omega:flat:Gevrey} and \eqref{eq:omega:flat:Gevrey} we similarly obtain  
\begin{align}
\int_0^t \norm{\pa_y \omega^\flat(s)}_{\gamma,r- \gamma + \frac 34,\tau(s)}^2 ds  + \beta   \int_0^t \norm{\omega^\flat(s)}_{\gamma,r - \gamma + \frac 54 ,\tau(s)}^2 ds  
&\les
\frac{1}{\beta^{1/2}} \int_0^t \abs{h(s)}_{\gamma,r+\frac 12,\tau(s)}^2 ds
\notag\\
&\leq  \frac{1}{2 \delta_0^2}  \norm{\omi(0)}_{\gamma,r,\tau_0}^2 
\label{eq:aux:2}
\end{align}
as above.
Summing \eqref{eq:aux:1}--\eqref{eq:aux:2} with \eqref{eq:prop:omega:in:1} (and using $(a+b)^2 \le 2 a^2 + 2 b^2$) we obtain
\begin{align*}
&\sup_{s\in [0,t]} \norm{\omega(s)}_{\gamma,r-\gamma+\frac{3}{4},\tau(s)}^2 + \int_0^t \norm{\pa_y \omega(s)}_{\gamma,r-\gamma +  \frac 34,\tau(s)}^2 ds + \beta \int_0^t \norm{\omega(s)}_{\gamma,r-\gamma+\frac 54,\tau(s)}^2 ds 
\notag\\
&\qquad \leq \frac{4}{\delta_0^2}  \norm{\omi(0)}_{\gamma,r,\tau_0}^2  
\end{align*}
by using that $\gamma \leq 5/4$.
This concludes the proof of \eqref{eq:omega:Gevrey:0}.
\end{proof}
\noindent
As an easy consequence of the estimate \eqref{eq:omega:Gevrey:0}, we state:
\begin{corollary}
\label{cor:omega:Gevrey}
Let $M, \delta_0$ and $\gamma \in [1,9/8]$ be given. For $r \ge r_0(\gamma)$,  $\beta \ge \beta_0$ and 
$T$ such that $\tau(t) \ge \tau_1$ for all $t \in [0,T]$, if 
\begin{equation} \label{eq:cond:omega}
\frac{4}{\delta_0^2} \|  \omega_0 \|_{\gamma,r,\tau_0} \le \frac{M}{2} 
\end{equation}
then 
$$ \sup_{t \in [0,T]} \|  \omega(t) \|_{\gamma,\frac{3r}{4},\tau(t)} \le \frac{M}{2}. $$
\end{corollary}

\section{Estimates for $\pa_t \omega$}
In order to emphasize the linear nature of the estimates in this section we denote
$\pa_t \omega = \dot \omega$. The equation obeyed by $\dot{\omega}$ is 
\begin{subequations}
\label{eq:dt:omega}
\begin{align}
 \pa_t \dot{\omega} - \pa_y^2\dot{\omega} + (u \pa_x + v \pa_y) \dot{\omega} + (\dot u \pa_x + \dot v \pa_y) \omega = 0 & \label{eq:dt:omega:a}\\
 \pa_y\dot{\omega}|_{y=0,1} = (\tilde{\dot \omega}\vert_{y=1} - \tilde{\dot\omega}\vert_{y=0})  - \pa_x \left( 2\int_0^1 u \, \dot{u} \; dy\right). & \label{eq:dt:omega:b}
\end{align}
\end{subequations}

\begin{proposition}
\label{prop:omega:dot}
Let $M, \delta_0$ and $\gamma \in [1,9/8]$ be given. There exists $r_1 = r_1(\gamma) \ge r_0$ such that: for all $r,r'$ satisfying $r' \ge r_1$,  $\frac{3r}{4} - r' \ge r_1$, one can find $\beta_1 = \beta_1(M,\delta_0,\tau_0,\tau_1,r,r',\gamma) \ge \beta_0$ satisfying: if $\beta \geq \beta_0$, if $T \leq 1$ small enough so that $\tau(t) \ge \tau_1$ for all $t\in [0,T]$, and if \eqref{eq:omi:ass:1}--\eqref{eq:omi:ass:3} hold, we have
\begin{align}
& \sup_{s\in[0,t]} \norm{\dot{\omega} (s)}_{\gamma,r'-\gamma+\frac{3}{4},\tau(s)}^2  + \int_0^t \norm{\pa_y \dot{\omega}(s)}_{\gamma,r'-\gamma+\frac{3}{4},\tau(s)}^2 ds + \beta \int_0^t \norm{\dot{\omega} (s)}_{\gamma,r'-\gamma+\frac{5}{4},\tau(s)}^2 ds  \notag\\
&\qquad   \leq  \frac{4}{\delta_0^2} \norm{\dot{\omega}(0)}_{\gamma,r',\tau_0}^2.
\label{eq:omega:Gevrey:1}
\end{align}
\end{proposition}
\begin{proof}[Proof of Proposition~\ref{prop:omega:dot}]
The proof  is very similar to that of Proposition~\ref{prop:omega:in}, since one may view equation \eqref{eq:dt:omega} as linearizing about $\omega$ itself of \eqref{eq:HNS:vort} (respectively $u$ for the boundary condition). In order to avoid redundancy, we only emphasize the essential differences.

\mspace
Estimate \eqref{eq:omega:Gevrey:1} follows directly from estimates for ${\dot \omega}^{ in}$ which are analogous to \eqref{eq:prop:omega:in:1}. In order to define ${\dot \omega}^{in}$, we define ${\dot \omega}^{\flat}$ as the solution of system \eqref{eq:flat} with boundary datum given by $\pa_x \dot h = - 2 \pa_x \int_0^1 u \, \dot u \; dy$, which is consistent with \eqref{eq:dt:omega:b}. The function ${\dot \omega}^{\flat}$ obeys all the estimates claimed in Lemma~\ref{lem:BL:1}, except that on the right side we need to replace $h$ with $\dot h$. As in \eqref{eq:bl:def} we define the boundary layer functions corresponding to $\dot \omega$, and according to \eqref{eq:in:def} we define the interior functions corresponding to $\dot \omega$. Note that as before we impose ${\dot \omega}^{bl}(0) =0$, and thus ${\dot \omega}^{in}(0) = {\dot \omega}_0$, where by  \eqref{eq:HNS:vort:evo}:
\begin{equation*}
\dot{\omega}_0 = -u_0 \pa_x \omega_0 - v_0 \pa_y \omega_0 - \pa^2_y \omega_0.  
\end{equation*}

\mspace
At this stage, we can prove an analogous statement to the one provided by Lemma~\ref{lem:h:omega:in}, with $h$ being replaced by 
$$\dot h  =  2\int_0^1 u \, \dot{u} \; dy  - 2 \int_\T \int_0^1  u \, \dot{u} \; dy dx.$$ 
Namely, we can show that for any $r$ as in Proposition \ref{prop:omega:in} and any $r'$ such that 
$$\frac{3r}{4} - \frac{\gamma}{2} - 1 \ge r' > 2\gamma+2, $$
we have
\begin{align}
\int_0^t \abs{\dot h(s)}_{\gamma,r',\tau(s)}^2 ds 
\les M^2\int_{0}^{t} \norm{{\dot \omega}^{in}(s)}_{\gamma,r',\tau(s)}^2 ds. \label{eq:dt:omega:in:j:6}
\end{align}
Indeed, denoting for all $f$
$$
f'_j = (j+1)^{r'-r} f_j = M'_j \pa^j_x f, \quad \mbox{where} \quad M'_j  = \frac{(j+1)^{r'} \tau^{j+1}}{(j!)^\gamma},
$$
similarly to \eqref{eq:h0:estimate} we obtain $\norm{\dot h_0}_{L^2_x} \les \norm{\dot h_1}_{L^2_x }$, while for $j \ge 1$,  as a substitute to \eqref{estim_hj} we obtain the inequality 
\begin{align*} 
\norm{\dot{h}'_j}_{L^2_x} 
&\les \sum_{\ell=1}^{j} {j\choose \ell}  \frac{M'_j}{M'_{j-\ell} M_{\ell}^{'1/2} M_{\ell+1}^{'1/2}}  \norm{\omega'_\ell}_{L^2_{x,y}}^{1/2}   \norm{\omega'_{\ell+1}}_{L^2_{x,y}}^{1/2}   \left( \norm{\dot{\omega}_{j-\ell}^{in'}}_{L^2_{x,y}}  + \norm{y(1-y) \dot{u}_{j-\ell}^{bl'}}_{L^2_{x,y}} \right)\\
&\qquad + M \left(\norm{\dot \omega_j^{in}}_{L^{2}_{x,y}} + \norm{y \dot u_j^\flat}_{L^2_{x,y}} + \norm{y \dot \omega_j^\flat}_{L^2_{x,y}}  +\norm{\dot u_j^{\flat}(x,1/2)}_{L^2_x} \right).
\end{align*}
The half sum $\sum_{\ell=1}^{\lceil j/2 \rceil}$ and the last term at the right-hand side can be treated as before, resulting in 
\begin{align*}
 & \int_0^t  \biggl( \sum_{\ell=1}^{\lceil j/2 \rceil}{j\choose \ell}  \dots  \: + \: M \Bigl(\norm{\dot \omega_j^{in}}_{L^{2}_{x,y}} + \dots +  \norm{\dot u_j^{\flat}(x,1/2)}_{L^2_x} \Bigr) \biggr)^2 \\
 &  \lesssim M^2 \left(\int_0^t \norm{\dot{\omi}(s)}_{\gamma,r',\tau(s)}^2 ds   + \frac{1}{\beta^{5/2}} \int_0^t \abs{\dot{h}(s)}_{\gamma,r'+\gamma- \frac 54,\tau(s)}^2 ds \right)
\end{align*}
if  $\sup_{t \in [0,T]}  \|  \omega(t) \|_{\gamma,\frac{r'}{4},\tau(t)}  \le M$, which is satisfied by assumption \eqref{eq:omi:ass:1} as soon as $r' \le 3r$. 

\mspace
For the half-sum $\sum_{\ell=\lceil j/2 \rceil+1}^{j}$, we can not proceed symmetrically as in the proof of Lemma~\ref{lem:h:omega:in}: as we want an $L^2$ in time control by $\dot{\omega}$, the bound
$$ {j\choose \ell}  \frac{M'_j}{M'_{j-\ell} M_{\ell}^{'1/2} M_{\ell+1}^{'1/2}} \lesssim (l+1)^{\gamma/2} $$
yields by a discrete convolution inequality:]
\begin{align*}
\int_0^t  \Bigl(\sum_{\ell=\lceil j/2 \rceil+1}^{j}  \dots \Bigr)^2  \lesssim  \left( \sup_{[0,t]}  \sum_{\ell \geq 1}  (\ell+1)^{\frac{\gamma}{2}}  \|\omega'_\ell\|_{L^2} \right)^2 \, \int_0^t \left(  \norm{\dot{\omi}(s)}_{\gamma,r',\tau(s)}^2 + \| y \dot{u}^\flat(s) \|_{\gamma,r',\tau(s)}^2 \right) ds  
 \end{align*}
Writing   $\sum_{\ell}  (\ell+1)^{\frac{\gamma}{2}}  \|\omega'_\ell\|_{L^2} =  \sum_{\ell} \frac{1}{\ell+1}  \left( (\ell+1)^{\frac{\gamma}{2}+1}  \|\omega'_\ell\|_{L^2}\right) $ and using Cauchy-Schwartz, we find:
 \begin{align*}
& \int_0^t  \Bigl(\sum_{\ell=\lceil j/2 \rceil+1}^{j}  \dots \Bigr)^2  \\
&  \lesssim \sup_{[0,t]}    \|  \omega(s) \|^2_{\gamma,r'+\frac{\gamma}{2}+1,\tau(s)}  \left(\int_0^t \norm{\dot{\omi}(s)}_{\gamma,r',\tau(s)}^2 ds   + \frac{1}{\beta^{7/2}} \int_0^t \abs{\dot{h}(s)}_{\gamma,r'+\gamma- \frac{7}{4},\tau(s)}^2 ds \right) \\
 & \lesssim  M^2  \left(\int_0^t \norm{\dot{\omi}(s)}_{\gamma,r',\tau(s)}^2 ds   + \frac{1}{\beta^{7/2}} \int_0^t \abs{\dot{h}(s)}_{\gamma,r'+\gamma-\frac{7}{4},\tau(s)}^2 ds \right)
 \end{align*}
where the last inequality comes from \eqref{eq:omi:ass:1}, under the assumption that $r'+\frac{\gamma}{2} +1 \le \frac{3r}{4}$. Gathering the two previous inequalities yields  \eqref{eq:dt:omega:in:j:6} for $\beta$ sufficiently large.

\mspace
Now,  similarly to \eqref{eq:omega:in:j}, we have that 
\begin{subequations}
\label{eq:dt:omega:in}
\begin{align}
&(\partial_t + \beta (j+1) - \pa_y^2) {\dot \omega}^{in'}_j + (u\pa_x +v \pa_y) {\dot \omega}^{in'}_j + {\dot v}^{in'}_j \pa_y \omega 
\notag\\
&\qquad    = - (u\pa_x +v \pa_y) {\dot \omega}^{bl'}_j - {\dot v}_j^{bl'} \pa_y \omega - M'_j \left[ \pa_x^j, u\pa_x + v\pa_y\right] \dot{\omega}
- M'_j \pa_x^j (\dot u \pa_x \omega) - M'_j \left[ \pa_x^j , \pa_y \omega \right] \dot{v}
\label{eq:dt:omega:in:a}
\\
 &\pa_y {\dot \omega}^{in}_j\vert_{y=0,1}   = {\tilde {\dot\omega}}^{in'}_j\vert_{y=1} - {\tilde{\dot\omega}}^{in'}_j\vert_{y=0}   +  2 {\dot \omega}^{\flat'}_j\vert_{y=1} - \pa_y {\dot \omega}^{\flat'}_j\vert_{y=1}.
 \label{eq:dt:omega:in:b}
\end{align}
\end{subequations}
Note that \eqref{eq:dt:omega:in:b} is the same as \eqref{eq:omega:in:j:b}, the left side of \eqref{eq:dt:omega:in:a} is the same as the left side of \eqref{eq:omega:in:j:a}, and the first two terms on the right side of \eqref{eq:dt:omega:in:a} are the same as the first two terms on the right side of \eqref{eq:omega:in:j:a}. The difference comes from the last three terms at the right-side of \eqref{eq:omega:in:j:a}, namely the quadratic terms. The main point is that they now  lack of symmetry: they  involve not only   
$(\dot{\omega}^{in'},\dot{\omega}^{bl'})$ but also $\omega$. In particular, all terms containing $\omega$ must be controlled uniformly in time, to allow for the $L^2_t$ control of $\dot{\omega}^{in'}$ at the left-hand side. This is why we take $r'$ less than $\frac{3r}{4}$ : with such a margin we can still use \eqref{eq:omi:ass:1} to control uniformly in time the  terms where most derivatives fall on $\omega$. 

\mspace
More precisely, proceeding as in the proof of \eqref{eq:dt:omega:in:j:6} to handle  the linear terms (see the estimates of $T_{1j}$, \dots,$T_{7j}$),   we can show that for $\beta$ large enough:
\begin{align}
&\norm{{\dot \omega}^{in}(t)}_{\gamma,r',\tau(t)}^2 +2  \beta \int_0^t   \norm{{\dot \omega}^{in}}_{\gamma,r'+1/2,\tau}^2  ds +  \frac{3}{2} \int_0^t \norm{\pa_y{\dot \omega}^{in}}_{\gamma,r',\tau}^2 ds - \frac{1}{\delta_0^2} \norm{{\dot \omega}_0}_{\gamma,r',\tau_0}^2 \notag\\
&\les     \frac{M^4}{\delta_0^7}   \int_0^t \norm{{\dot \omega}^{in}}_{\gamma,r',\tau}^2 ds +  \frac{M^2}{\delta_0 \beta^{3/4}}  \int_0^t \norm{{\dot \omega}^{in}}_{\gamma,r'+\frac 12,\tau}^2   ds
\notag\\
&\qquad  + \sum_{j \geq 0} \int_0^t \left( S_{1j} + S_{2j} + S_{3j} + S_{4j} \right)(s) ds,
\label{eq:omegadot:in:j:3}
\end{align}
 where 
\begin{align*}
 S_{1j}  & = - \int_{\Omega} M'_j [\pa^j_x , u \pa_x] \dot{\omega}  \frac{{\dot \omega}^{in'}_j}{\pa_y \omega}, \quad  
 S_{2j}   =  - \int_{\Omega}  M'_j [\pa^j_x , v \pa_y] \dot{\omega}  \frac{{\dot \omega}^{in'}_j}{\pa_y \omega} \\
 S_{3j}  & =  - \int_{\Omega}  M'_j \pa^j_x (\dot{u} \pa_x \omega)  \frac{{\dot \omega}^{in'}_j}{\pa_y \omega}, \quad
 S_{4j}   = - \int_{\Omega}   M'_j [\pa^j_x, \pa_y \omega]  \dot{v}  \frac{{\dot \omega}^{in'}_j}{\pa_y \omega}.
 \end{align*}
The first term is  analogue to $T_{8j}$. One can write 
\begin{align*}
 S_{1j} = - \left( \sum_{k=1}^{\lceil j/2\rceil} + \sum_{k=\lceil j/2\rceil+1}^j \right) \binom{j}{k} \frac{M'_j}{M'_k M'_{j-k+1}} \int_\Omega u_k' \dot{\omega}'_{j-k+1} \frac{{\dot \omega}^{in'}_j}{\pa_y \omega}
= S_{1j,{\rm low}} + S_{1j,{\rm high}}. 
\end{align*}
The treatment of $S_{1j,{\rm low}}$ is exactly the same as the one of $T_{8j,{\rm low}}$. Similarly to \eqref{eq:T8j:low}, \eqref{eq:T8j}, we get
$$  \sum \int_0^t S_{1j,{\rm low}}(s) ds \lesssim \frac{M^2}{\delta_0} \int_0^t  \|  \dot{\omega}^{in}(s) \|^2_{\gamma,r'+\frac{1}{2},\tau(s)} ds.   $$
To treat $S_{1j,{\rm high}}$, we use the inequality  $\displaystyle \binom{j}{k}  \frac{M'_j}{M'_k M'_{j-k+1}} \lesssim (j-k+1)^{\gamma-r'}$ for $k \ge \lceil j/2\rceil+1$, so that 
\begin{align*}
 S_{1j,{\rm high}}  &  \lesssim \sum_{k=\lceil j/2\rceil+1}^j  \frac{1}{\delta_0} \|  u'_k \|_{L^\infty}  (j-k+1)^{\gamma-r'} \|  \dot{\omega}^{'}_{j-k+1} \|_{L^2} \|  \dot{\omega}^{in'}_{j} \|_{L^2} \\
  & \lesssim  \sum_{k=\lceil j/2\rceil+1}^j  \frac{k^\gamma}{\delta_0} \|  \omega'_{k+1} \|_{L^2}  (j-k+1)^{\gamma-r'}  \|  \dot{\omega}^{'}_{j-k+1} \|_{L^2} \|  \dot{\omega}^{in'}_{j} \|_{L^2} 
 \end{align*}
so that by the discrete Young's inequality:
\begin{align*}
 \sum \int_0^t S_{1j,{\rm high}}(s) ds  
 &\lesssim  \frac{1}{\delta_0}  \sup_{s \in [0,t]} \sum_{k} k^\gamma \|  \omega'_k(s) \|_{L^2}  \int_0^t   \|  \dot{\omega}(s) \|_{\gamma, \gamma, \tau(s)}   \|  \dot{\omega}^{in} \|_{\gamma, r',\tau(s)}  \\
& \lesssim  \frac{1}{\delta_0}  \sup_{s \in [0,t]}  \|  \omega(s)\|_{\gamma,r'+\gamma+1,\tau(s)}\int_0^t   \|  \dot{\omega}(s) \|_{\gamma, \gamma, \tau(s)}   \|  \dot{\omega}^{in} \|_{\gamma, r',\tau(s)}  \end{align*}
 The sup in time is controlled as usual by assumption \eqref{eq:omi:ass:1}, under the constraint $r'+\gamma+1 \le \frac{3r}{4}$. As regards the second factor, one can split $\|  \dot{\omega}(s) \|_{\gamma, \gamma, \tau(s)} \le \|  \dot{\omega}^{in}(s) \|_{\gamma, \gamma, \tau(s)} +  \|  \dot{\omega}^{bl}(s) \|_{\gamma, \gamma, \tau(s)}$ and control the second term by the analogue of Lemma \ref{lem:BL:1}, followed by \eqref{eq:dt:omega:in:j:6}. For $r' \ge \gamma + (\gamma+\frac{3}{4})$ we find that
  $$  \sum \int_0^t S_{1j,{\rm high}}(s) ds   \lesssim \frac{M^2}{\delta_0} \int_0^t  \|  \dot{\omega}^{in}(s) \|^2_{\gamma,r',\tau(s)} ds. $$
  
 \mspace
 Estimates on $S_{2j}$ (which is analogue to $T_{9j}$) and $S_{3j}$ can be established in the same way. We find for $r'$  and $\frac{3r}{4}-r'$ large enough (with thresholds depending on $\gamma$): 
 \begin{equation*} 
 \sum_j \int_0^t S_{2j} \le \eta \int_0^t  \|  \pa_y \dot{\omega}^{in}(s) \|^2_{\gamma, \gamma+r', \tau(s)} ds  + \frac{C}{\eta} \frac{M^4}{\delta_0^2} \int_0^t   
 \|  \dot{\omega}^{in}(s)\|^2_{\gamma, \gamma+r'+\frac{1}{2}, \tau(s)} ds
 \end{equation*}
  $C > 0$, $\eta$ arbitrarily small, and 
  \begin{equation*} 
 \sum_j \int_0^t S_{3j} \le  \frac{M^2}{\delta_0}  \int_0^t    \|  \dot{\omega}^{in}(s)^2 \|_{\gamma, \gamma+r', \tau(s)} ds
 \end{equation*}
To handle  $S_{4j}$, we proceed slightly differently. We start with the decomposition 
\begin{align*}
 S_{4j} &  = - \left( \sum_{k=0}^{\lceil j/2\rceil} + \sum_{k=\lceil j/2\rceil+1}^{j-1} \right) \binom{j}{k} \frac{M'_j}{M'_k M'_{j-k}} \int_\Omega \pa_y \omega'_{j-k} \dot{v}'_{k}   \frac{{\dot \omega}^{in'}_j}{\pa_y \omega}  \\
& = S_{4j,{\rm low}} + S_{4j,{\rm high}}. 
\end{align*}
 $S_{4j,{\rm high}}$ can be treated similarly to $T_{9j,{\rm high}}$. We obtain, see \eqref{eq:T9j:high}: 
 \begin{align*}
 \sum_j \int_0^t S_{4j,{\rm high}} & \lesssim \frac{1}{\delta_0} \sup_{[0,t]}\|  \pa_y \omega \|_{\gamma, \frac{r'}{2}} \int_0^t \left( \| \dot{\omega}^{in}(s)\|_{\gamma,r'+\frac{1}{2},\tau(s)} + \| \dot{u}^\flat\|_{\gamma,r'+\frac{1}{2},\tau(s)} \right)  \| \dot{\omega}^{in}(s)\|_{\gamma,r'+\frac{1}{2},\tau(s)}  ds \\
 & \lesssim \frac{M^2}{\delta_0} \int_0^t  \| \dot{\omega}^{in}(s)\|_{\gamma,r'+\frac{1}{2},\tau(s)}^2 ds. 
 \end{align*}
 Here, we have used the Gevrey control of $\pa_y \omega$ given by \eqref{eq:omi:ass:1} to bound the first factor, and the analogue of Lemma \ref{lem:BL:1}  followed by  \eqref{eq:dt:omega:in:j:6} to control the boundary layer term in the second factor. As regards $S_{4j,{\rm low}}$, we integrate by parts in $y$. As $\dot{v}$ vanishes at the boundary, no boundary term appears, and we get 
\begin{align*}  S_{4j,{\rm low}} & =    \sum_{k=0}^{\lceil j/2\rceil} \binom{j}{k} \frac{M'_j}{M'_k M'_{j-k}} \int_\Omega \Bigl(  \omega'_{j-k}  \pa_y \dot{v}'_{k}  \frac{{\dot \omega}^{in'}_j}{\pa_y \omega}  -  \omega'_{j-k}  \dot{v}'_{k}  \frac{\pa^2_y\omega}{(\pa_y \omega)^2}{\dot \omega}^{in'}_j  +  \omega'_{j-k}   \dot{v}'_{k}  \frac{\pa_y {\dot \omega}^{in'}_j}{\pa_y \omega} \Bigr)\\ 
& = S_{4j,{\rm low},1} + S_{4j,{\rm low},2} + S_{4j,{\rm low},3}. 
\end{align*}
We can bound $S_{4j,{\rm low},1}$ with the same ideas as before. For $r'$ and $\frac{3r}{4} - r'$ large enough we have
$$ \int_0^t \sum_j S_{4j,{\rm low},1} \lesssim \frac{M^2}{\delta_0} \int_0^t  \| \dot{\omega}^{in}(s)\|_{\gamma,r'+\frac{1}{2},\tau(s)}^2 ds. $$ 
As regards $S_{4j,{\rm low},2}$ we start from the bound 
\begin{align*} S_{4j,{\rm low},2}  & \lesssim  \frac{1}{\delta_0^2} \sum_{k=0}^{\lceil j/2\rceil}   \|  \omega'_{j-k} \|_{L^\infty_x L^2_y}  (k+1)^{-r'} \|   \dot{v}'_{k}  \|_{L^\infty} \|  \pa^2_y \omega \|_{L^\infty_x L^2_y}   \|  \dot{\omega}^{in'}_j \|_{L^2_x L^\infty_y}\\
& \lesssim   \frac{M}{\delta_0^2} \sum_{k=0}^{\lceil j/2\rceil}   \|  \omega'_{j-k} \|_{L^\infty_x L^2_y}  (k+1)^{-r'} \|   \dot{v}'_{k}  \|_{L^\infty}  \|  \dot{\omega}^{in'}_j \|_{L^2_x L^\infty_y}
\end{align*}
where the last inequality comes from \eqref{eq:omi:ass:3} to control $\pa^2_y \omega$. It follows that 
$$ S_{4j,{\rm low},2}  \lesssim    \frac{M}{\delta_0^2} \sum_{k=0}^{\lceil j/2\rceil}  (j-k+1)^{\gamma}  \|  \omega'_{j-k+1} \|_{L^2}  (k+1)^{-r'+2\gamma} \|   \dot{u}'_{k+2}  \|_{L^2}   (\| \dot{\omega}^{in'}_j \|_{L^2}  + \| \pa_y \dot{\omega}^{in'}_j \|_{L^2} ).     $$ 
From there,   for $r'$ and $\frac{3r}{4} - r'$ large enough (with thresholds depending on $\gamma$), 
$$  \int_0^t \sum_j S_{4j,{\rm low},2} \le \eta \int_0^t  \|  \pa_y \dot{\omega}^{in}(s) \|^2_{\gamma, \gamma+r', \tau(s)} ds  + \frac{C}{\eta} \frac{M^6}{\delta_0^4} \int_0^t   \|  \dot{\omega}^{in}(s)\|^2_{\gamma, \gamma+r', \tau(s)} ds.  $$ 
With similar manipulations, we  get the bound 
$$  \int_0^t \sum_j S_{4j,{\rm low},3} \le \eta \int_0^t  \|  \pa_y \dot{\omega}^{in}(s) \|^2_{\gamma, \gamma+r', \tau(s)} ds  + \frac{C}{\eta} \frac{M^4}{\delta_0^2} \int_0^t   \|  \dot{\omega}^{in}(s)\|^2_{\gamma, \gamma+r', \tau(s)} ds.  $$ 
Injecting the previous estimates in \eqref{eq:omegadot:in:j:3},  we get for large enough $\beta$: 
\begin{align*}
&\norm{{\dot \omega}^{in}(t)}_{\gamma,r',\tau(t)}^2 +   \beta \int_0^t   \norm{{\dot \omega}^{in}}_{\gamma,r'+1/2,\tau}^2  ds +  \int_0^t \norm{\pa_y{\dot \omega}^{in}}_{\gamma,r',\tau}^2 ds \le   \frac{1}{\delta_0^2}  \norm{{\dot \omega}_0}_{\gamma,r',\tau_0}^2.
\end{align*}
Estimate \eqref{eq:omega:Gevrey:1} follows from this inequality, in the same way as \eqref{eq:omega:Gevrey:0} is deduced from  \eqref{eq:prop:omega:in:1}.
\end{proof}

\begin{corollary}
\label{cor:dy:omega:Gevrey}
Let $M, \delta_0$ and $\gamma \in [1,9/8]$ be given. There exists  $r_2 = r_2(\gamma) \ge r_1$ such that for $r \ge r_2(\gamma)$,  one can find 
$\beta_2 = \beta_2(M,\delta_0,\tau_0,\tau_1,\gamma,r) \ge \beta_1$ and 
$$T_0 = T_0\left(M,\delta_0,\beta,\tau_0,\tau_1,\gamma,r,\norm{{\dot \omega}_0}_{\gamma,\frac r2+\gamma-\frac 34,\tau_0}\right) > 0 $$
 satisfying: if $\beta \geq \beta_0$, if $T \le T_0$, if \eqref{eq:omi:ass:1}-\eqref{eq:omi:ass:2}-\eqref{eq:omi:ass:3} hold, and if
\begin{align} \label{eq:dy:omega:cor}
\norm{ \pa_y\omega_0}_{\gamma, \frac r2 ,\tau_0} \leq \frac{M}{4},
\end{align}
then
\begin{align}
 \sup_{t\in[0,T]} \norm{\pa_y \omega(t)}_{\gamma,\frac r2,\tau(t)} \leq \frac{M}{2}.
\label{eq:nonsense} 
\end{align}
\end{corollary}
\begin{proof}[Proof of Corollary~\ref{cor:dy:omega:Gevrey}]
We write $ \pa_y \omega(t) =  \pa_y \omega_0 +  \int_0^t \pa_y \dot{\omega}(s) ds$, so that for all $t \in [0,T]$:
\begin{align*}
 \|  \pa_y \omega(t) \|_{\gamma,\frac r2,\tau(t)} & \le  \|  \pa_y \omega_0 \|_{\gamma,r/2,\tau(t)} + \int_0^t \|  \pa_y \dot{\omega}(s) \|_{\gamma, \frac r2,\tau(t)} ds \\
  & \le   \|  \pa_y \omega_0 \|_{\gamma, \frac r2,\tau(0)} +   \int_0^t \|  \pa_y \dot{\omega}(s) \|_{\gamma, \frac r2,\tau(s)} ds  \\ 
  & \le   \|  \pa_y \omega_0 \|_{\gamma,\frac r2,\tau(0)}  + \sqrt{t} \, \left( \int_0^t\|  \pa_y \dot{\omega}(s) \|^2_{\gamma, \frac r2,\tau(s)} ds \right)^{1/2}. 
\end{align*}
Taking for instance $r_2 = 4r _1+4\gamma+3$, where $r_1$ was introduced in  Proposition~\ref{prop:omega:dot}, and $r \ge  r_2$,  we ensure that $r' := \frac{r}{2}+\gamma-3/4$ satisfies $r' \ge r_1$ and $\frac{3r}{4} - r' \ge r_1$. By Proposition~\ref{prop:omega:dot}, for $\beta \ge \beta_0$ large enough, and $T$ such that $\tau(t) \in [\tau_1,\tau_0]$ for all $t \in [0,T]$, we get
\begin{equation} \label{bound:y:omega}
 \sup_{t \in [0,T]} \|  \pa_y \omega(t) \|_{\gamma,r/2,\tau(t)}  \le \|  \pa_y \omega_0 \|_{\gamma, \frac r2,\tau(0)}  + \frac{2\sqrt{T}}{\delta_0} \| \dot{\omega}(0)\|_{\gamma, \frac r2+\gamma-\frac 34,\tau_0}.  
 \end{equation}
The result follows from the assumption on $\pa_y \omega_0$, once $T_0$ is taken small enough to ensure that $\frac{2\sqrt{T_0}}{\delta_0} \|\dot{\omega(0)}\|_{\gamma,\frac r2+\gamma-\frac 34,\tau_0} \le \frac{M}{4}$ holds.
\end{proof}

\begin{corollary}
\label{cor:dy:dy:omega}
Let $M, \delta_0$ and $\gamma \in [1,9/8]$ be given.   There exists $r_3 = r_3(\gamma)\ge r_2$ such that for $r \ge r_3(\gamma)$,  one can find 
$\beta_3 = \beta_3(M,\delta_0,\tau_0,\tau_1,\gamma,r) \ge \beta_2$, $c_0 = c_0(\tau_0,\tau_1,\gamma,r) > 0 $ and 
\begin{equation} 
\label{eq:T:dependence}
T_0 = T_0\left(M,\delta_0,\beta,\tau_0,\tau_1,\gamma,r,\norm{\omega(0)}_{\gamma,r,\tau_0},\norm{{\dot \omega}(0)}_{\gamma,\frac r2+\gamma-\frac{3}{4},\tau_0}\right)>0
\end{equation}
 satisfying: if $\beta \geq \beta_0$, if $T \le T_0$, if \eqref{eq:omi:ass:1}-\eqref{eq:omi:ass:2}-\eqref{eq:omi:ass:3} hold, and if
 \begin{align}
\frac{1}{\delta_0} \|  \dot{\omega}_0 \|_{\gamma, \frac r2+\gamma-\frac{3}{4},\tau_0} + \frac{1}{\delta_0^2} \|  \omega_0 \|^2_{\gamma,r,\tau_0} + \frac{1}{\delta_0} \|  \omega_0 \|_{\gamma,r, \tau_0} \|  \pa_y \omega_0 \|_{\gamma, \frac r2,\tau_0}    &\leq \frac{c_0 M}{4},
 \label{eq:dy:dy:omega:cor:0}
\end{align} 
then
\begin{align*}
  \sup_{t\in[0,T]}  \norm{\pa_y^2  \omega(t)}_{L^\infty_x L^2_y} &\leq \frac{M}{2}.
\end{align*} 
\end{corollary}
\begin{proof}[Proof of Corollary~\ref{cor:dy:dy:omega}]
We write the vorticity equation under the form
$$ 
\pa^2_y \omega = \dot{\omega} + u \pa_x \omega  + v \pa_y \omega. 
$$
Hence, for all $t \in [0,T]$: 
$$ 
\|  \pa^2_y \omega(t) \|_{L^\infty_x L^2_y} 
\le \|  \dot{\omega}(t) \|_{L^\infty_x L^2_y}  
+  \|  u(t) \|_{L^\infty_{x,y}} \|  \pa_x \omega(t) \|_{L^\infty_x L^2_y} 
+  \|  v(t) \|_{L^\infty_{x,y}} \|  \pa_y \omega(t) \|_{L^\infty_x L^2_y}. $$
For $r$ large enough, we obtain
$$ \|  \pa^2_y \omega(t) \|_{L^\infty_x L^2_y}  \lesssim \|  \dot{\omega}(t) \|_{\gamma,\frac r2,\tau(t)} +  \|  \omega(t) \|^2_{\gamma,r-\gamma+\frac{3}{4},\tau(t)} + \|  \omega(t) \|_{\gamma,r-\gamma+\frac{3}{4},\tau(t)} \|  \pa_y \omega(t) \|_{\gamma, \frac r2,\tau(t)}. $$ 
By Propositions \ref{prop:omega:in} and Proposition \ref{prop:omega:dot} applied respectively with $r$ and $r' = \frac{r}{2} + \gamma - \frac{3}{4}$,  and by  inequality \eqref{bound:y:omega}, we find
\begin{align*} 
\sup_{t \in [0,T]} \|  \pa_y^2 \omega(t) \|_{L^\infty_x L^2_y}  
& \lesssim \frac{1}{\delta_0} \|  \dot{\omega}_0 \|_{\gamma,\frac r2+\gamma-\frac{3}{4},\tau_0} + \frac{1}{\delta_0^2} \|  \omega_0 \|^2_{\gamma,r,\tau_0} \\  
&\qquad  + \frac{1}{\delta_0} \|  \omega_0 \|_{\gamma,r, \tau_0} \left( \|  \pa_y \omega_0 \|_{\gamma,\frac r2,\tau_0}  + \frac{\sqrt{T}}{\delta_0} \|  \dot{\omega}_0\|_{\gamma,\frac r2+\gamma-\frac 34,\tau_0}\right).
\end{align*}
Upon taking $T$ sufficiently small, this concludes the proof of the Corollary. 
\end{proof}

\section{Minimum and maximum principle for $\pa_y \omega$} \label{sec:max}
The quantity $\pa_y \omega$ obeys a (degenerate) parabolic equation with Dirichlet boundary conditions
\begin{subequations}
\begin{align}
 \pa_t (\pa_y \omega) - \pa_y^2(\pa_y\omega) + (u \pa_x + v \pa_y) (\pa_y \omega) + (\pa_x u) (\pa_y \omega) = \omega \pa_x \omega & \\
 \pa_y \omega|_{y=0,1} = (\tilde \omega\vert_{y=1} - \tilde \omega\vert_{y=0}) -  \pa_x \int_0^1 u^2 dy . &
\end{align}
\end{subequations}
Our goal is to combine this fact with $L^2_t L^\infty_{x,y}$ estimates on $\omega \, \pa_x \omega$ and the Dirichlet datum, to deduce that the convexity of $u$ is conserved for small time.

\begin{proposition}
\label{prop:maximum:principle}
Let $M, \delta_0>0$ and $\gamma \in [1,9/8]$ be given. There exists  $r_4 = r_4(\gamma) \ge r_3$ such that for $r \ge r_4(\gamma)$,  one can find 
$\beta_4 = \beta_4(M,\delta_0,\tau_0,\tau_1,\gamma,r) \ge \beta_3$ and $T_0$ as in \eqref{eq:T:dependence}
 satisfying: if $\beta \geq \beta_0$, if $T \le T_0$, if \eqref{eq:omi:ass:1}-\eqref{eq:omi:ass:2}-\eqref{eq:omi:ass:3} hold, and if
\begin{align}
 4 \delta_0 \leq \pa_y  \omega_0  &\leq \frac{1}{4\delta_0},
 \label{eq:convex:IC}
\end{align} 
then
\begin{align}
 2 \delta_0 \leq \pa_y  \omega(t)  &\leq \frac{1}{2\delta_0}, \quad \forall t\in [0,T].
 \label{eq:convex:T}
\end{align}  
\end{proposition}
\begin{proof}[Proof of Proposition~\ref{prop:maximum:principle}]
We wish to apply a version of the parabolic minimum/maximum principle for the following degenerate parabolic problem posed in $\Omega \times (0,T)$, with $\Omega$ being the periodic in $x$ strip $(x,y) \in \T \times (0,1)$:
\begin{subequations}
\begin{align}
(\partial_t - \pa_y^2 + b(x,y,t) \cdot \nabla_{x,y} + c(x,y,t)) \psi &= d(x,y,t)  &&\mbox{in} \quad \Omega \times (0,T), 
\label{eq:parabolic:a}\\
\psi  &= a(x,t)  &&\mbox{on} \quad \partial\Omega \times [0,T),
\label{eq:parabolic:b}
\\
\psi|_{t=0} &= \psi_0(x,y) &&\mbox{in} \quad \Omega.
\label{eq:parabolic:c}
\end{align}
\end{subequations}
Here $\psi = \pa_y \omega$, $b = (u,v)$ is incompressible and vanishes on the boundary $\T \times \{0,1\}$, $c = \pa_x u$   vanishes at the boundary $\T \times \{0,1\}$,  $d = \omega \pa_x \omega$, and the boundary data is  $a =  (\tilde \omega\vert_{y=1} - \tilde \omega\vert_{y=0})  -  \pa_x \int_0^1 u^2 dy $. As emphasized after Theorem \ref{thm:main}, the third compatibility condition of the theorem corresponds to the relation $a(x,0) = \psi_0(x,0)$.

\mspace
By \eqref{eq:convex:IC}, the initial datum $\psi_0$ is taken to obey $0< 4\delta_0 \leq \psi_0(x,y) \leq 1/(4\delta_0)$, for some $\delta_0 \in (0,1/4)$, uniformly on $\Omega$. Thus, by the compatibility of the initial datum and of the boundary condition, we have that $0< 4\delta_0 \leq a(x,0) \leq 1/(4\delta_0)$, uniformly on $\T$. Thanks to the Gagliardo-Nirenberg inequality 
$$ \|  f \|_{L^\infty_y} \le C \|  f\|_{L^2_y}^{1/2} \left(  \|  f\|_{L^2_y}^{1/2} + \| \pa_y f\|_{L^2_y}^{1/2} \right)$$ 
and the estimate  \eqref{eq:omega:Gevrey:1}, we  have that 
\begin{align*}
\norm{\partial_t a(x,t)}_{L^2(0,T;L^\infty_x)} 
& \le 4 \norm{\dot{\omega}}_{L^2(0,T; L^\infty)} +  2\norm{\pa_x \int_0^1 u \, \dot u \, dy}_{L^2(0,T;L^\infty_x)} \notag \\
& \lesssim   \frac{1}{\delta_0^2}  \left( \frac{1}{\beta^{1/4}} + \frac{M}{\beta^{1/2}} \right)  \norm{{\dot \omega_0}}_{\gamma, \frac r2+\gamma-\frac{3}{4},\tau_0} 
\leq  \norm{{\dot \omega_0}}_{\gamma, \frac r2+\gamma-\frac{3}{4},\tau_0} 
\end{align*}
 for $\beta$ sufficiently large. By the fundamental theorem of calculus in time, and the Cauchy-Schwartz inequality we thus obtain that 
\begin{align*}
3\delta_0 \leq 4 \delta_0 - \sqrt{T}\norm{{\dot \omega_0}}_{\gamma, \frac r2+\gamma-\frac{3}{4},\tau_0}  \leq a(x,t) \leq \frac{1}{4\delta_0} +  \sqrt{T}\norm{{\dot \omega_0}}_{\gamma, \frac r2+\gamma-\frac{3}{4},\tau_0} \leq \frac{1}{3\delta_0}
\end{align*}
uniformly on $\T \times (0,T)$, upon taking $T$ sufficiently small. Thus, on the parabolic boundary $\Omega \times \{0\} \cup \partial\Omega \times (0,T)$, we have that $\psi \geq 3\delta_0$.

By the same Gagliardo-Nirenberg inequality, the Poincar\'e inequality in $y$, and estimate \eqref{eq:omega:Gevrey:0}, we have 
\begin{align*}
\sup_{t\in [0,T]} \norm{c(t)}_{L^\infty_x L^\infty_y} =  \sup_{t\in [0,T]} \norm{\pa_x u(t)}_{L^\infty_x L^\infty_y} \leq \frac{C_1}{\delta_0} \norm{\omega_0}_{\gamma,r,\tau_0}
\end{align*}
where $C_1  = C_1(\tau_0,\tau_1,\gamma,r)$. Denoting
\begin{align}
 C_* = 1+ \frac{C_1}{\delta_0} \norm{\omega_0}_{\gamma,r,\tau_0},
 \label{eq:C*}
\end{align}
the above estimate implies that 
\begin{align*}
c(x,y,t) + C_*  \geq 1.
\end{align*}
Lastly, we note that by the Gagliardo-Nirenberg inequality and \eqref{eq:omega:Gevrey:0} we have
\begin{align*}
\int_0^t \norm{d(s)}_{L^\infty_x L^\infty_y} ds
&= \int_0^t \norm{\omega(s)}_{L^\infty_x L^\infty_y} \norm{\pa_x \omega(s)}_{L^\infty_x L^\infty_y} ds 
\les  \frac{\sqrt{t}}{\delta_0^2} \norm{\omega_0}_{\gamma,r,\tau_0}^2
\end{align*}
so that for $T \le 1$ we have
\begin{align}
e(t) &:=t+ \int_0^t e^{-C_* s} \norm{d(s) - 3\delta_0 c(s)}_{L^\infty_x L^\infty_y} ds 
\notag\\
&\les t+ \sqrt{t} \norm{\omega_0}_{\gamma,r,\tau_0}^2 + t C_1  \norm{\omega_0}_{\gamma,r,\tau_0}
\notag\\
&\leq C_2 \sqrt{t} \left(1 + \norm{\omega_0}_{\gamma,r,\tau_0}^2 +    \norm{\omega_0}_{\gamma,r,\tau_0}\right)
= \sqrt{t} D_* 
\label{eq:D*}
\end{align}
holds for all $t \in [0,T]$, where $C_2$ is a constant that only depends on $\gamma,r,\tau_0$, and $\tau_1$, and we have denoted
\begin{align*}
 D_* = C_2  \left(1 + \norm{\omega_0}_{\gamma,r,\tau_0}^2 +   \norm{\omega_0}_{\gamma,r,\tau_0}\right).
\end{align*}
With this notation, we make the following change of unknowns 
\begin{subequations}
\begin{align}
\bar \psi &= e^{-  C_* t}(\psi(x,y,t) - 3\delta_0) + e(t) \\
\bar a &= e^{-  C_* t} (a(x,t) - 3\delta_0) + e(t)   \\
\bar d &= e^{-  C_* t}(d(x,y,t) - 3\delta_0 c(x,y,t)) \\
\bar c &= c(x,y,t) +  C_*  \\
\bar \psi_0 &= \psi_0(x,y) - 3 \delta_0 
\end{align}
\end{subequations}
The quantity $e(t)$ was chosen so that  $\dot e(t) = 1 + \norm{\bar d(t)}_{L^\infty} $. One may then verify directly that 
\begin{subequations}
\label{eq:bar:psi}
\begin{align}
(\partial_t - \pa_y^2 + b\cdot \nabla_{x,y} + \bar c) \bar \psi 
&= \left( \bar d + \norm{\bar d}_{L^\infty} \right) + 1 + \bar c e  \geq 1 > 0 \\
\bar \psi|_{y \in \{0,1\}} &= \bar a  \geq t \geq 0\\
\bar \psi|_{t= 0 } &= \bar \psi_0  \geq \delta_0 > 0.
\end{align}
\end{subequations}
The parabolic minimum principle then guarantees that 
\begin{align}
 \bar \psi(x,y,t) \geq 0 \quad \mbox{on} \quad \Omega \times [0,T]
 \label{eq:minimum}
\end{align}
Indeed, if  a strictly negative minimum would be attained by $\bar \psi$, then this point minimum could not lie on the parabolic boundary (since $\bar a \geq 0$ and $\bar\psi_0 > 0$). If this point would lie in the interior, at this point we would need to have $\nabla_{t,x,y} \bar\psi = 0$, whereas $(-\pa_y^2 + \bar c)\bar \psi < 0$ since $\bar c >0$. This contradicts  $\left( \bar d + \norm{\bar d}_{L^\infty} \right) + 1 + \bar c e > 0$, which thus proves \eqref{eq:minimum}.

Working backwards from the definition of $\bar \psi$, we see that \eqref{eq:C*}, \eqref{eq:D*}, and \eqref{eq:minimum} imply
\begin{align*}
 \psi(x,y,t) \geq 3\delta_0 - e^{C_*t} e(t) \geq 3 \delta_0 -  \sqrt{T} e^{C_* T} D_* \geq 2 \delta_0
\end{align*}
as long as  $T$ is chosen sufficiently small in terms of $C_*,D_*$, and $\delta_0$, consistent with the dependence given in \eqref{eq:T:dependence}. This proves the lower bound in \eqref{eq:convex:T}.

The proof of the upper bound in \eqref{eq:convex:T} follows from very similar arguments, reducing the problem to a maximum principle for a parabolic equation. To avoid redundancy, we omit these details.
\end{proof}

\section{Proof of Theorem~\ref{thm:main}} \label{sec:proof:main}
The proof of the main theorem proceeds as follows. Let $\gamma \leq 9/8$ and $r \ge r_4(\gamma)$.  For any  $ \tau_0 < \tau^0$ assumption \eqref{eq:thm:main:1} implies that  $\omega_0 = \pa_y u_0$ satisfies 
\begin{align*}
\|  \omega_0 \|_{\gamma,r,\tau_0}  +  \|  \pa^2_y \omega_0 \|_{\gamma,r,\tau_0} < +\infty.
\end{align*} 
We fix $\tau_0 \in  (\tau_1, \tau^0)$. We then fix   $\delta_0$ small enough and $M$ large enough,  so that the initial constraints  \eqref{eq:cond:omega}, \eqref{eq:dy:omega:cor}, \eqref{eq:dy:dy:omega:cor:0} and \eqref{eq:convex:IC} hold. Let $\beta \ge \beta_4$ and $\eps > 0$. We consider the approximate system 
\begin{subequations}
\label{eq:HNS:approx}
\begin{align}
\pa_t u + u \pa_x u + v \pa_y u  +  \pa_x p -  \pa^2_y u - \eps \pa^2_x u & = 0, \quad (x,y) \in \T \times (0,1), 
\label{eq:HNS:approx:evo}\\ 
\pa_y p & = 0, \quad (x,y) \in \T \times (0,1), \\
\pa_x u + \pa_y v & = 0, \quad (x,y) \in \T \times (0,1), \label{eq:HNS:approx:incompressible}\\
u\vert_{y=0,1} = v\vert_{y=0,1} & = 0,\label{eq:HNS:approx:BC}
\end{align}
\end{subequations}
with the same initial condition $u\vert_{t=0} = u_0$. System \eqref{eq:HNS:approx} is called the two-dimensional primitive equations, and has been widely studied,  in various geometries and under various boundary conditions~\cite{BreschKazhikhovLemoine04,BreschGuillenMadmoudiRodriguez03,TemamZiane04}. In particular, Gevrey or analytic regularity results were obtained in both periodic and bounded geometries~\cite{Petcu04,PetcuTemamWirosoetisno04,KukavicaLombardoSammartino16}. In the context of system \eqref{eq:HNS:approx}, the well-posedness result stated in Theorem~\ref{thm:main} can be proved without much difficulty. In fact, the presence of $-\eps \pa^2_x u$ allows for a classical treatment, and the existence of solutions at fixed $\eps>0$ follows e.g. from a Galerkin approximation procedure (which is compatible with the hydrostatic trick~\cite{MasmoudiWong12}).  Moreover, the compatibility conditions are the same for \eqref{eq:HNS} and \eqref{eq:HNS:approx}. We find in this way a unique  local solution $u^\eps$ with the regularity requirements stated in Theorem \ref{thm:main}.  We can then consider  $T_{\eps,*}$ the maximal time  on which  $\| \omega_\eps \|_{\gamma,0,\tau_1} < +\infty$. In particular, if $T_{\eps,*}$ is small enough so that $\tau(T_{\eps,*}) \ge  \tau_1$, one has  
\begin{equation}  \label{eq:blowup}
\sup_{t \in [0,T_{\eps,*})} \|  \omega_\eps(t) \|_{\gamma,\frac{3r}{4},\tau(t)} = +\infty .
\end{equation} 
By the initial constraint \eqref{eq:cond:omega},  the fact that  $\tau_0 < \tau^0$, and the continuity of the solution, there exists  a maximal  time  $0 < T_\eps   \le T_{\eps,*}$ on which the conditions  \eqref{eq:omi:ass:1}-\eqref{eq:omi:ass:2}-\eqref{eq:omi:ass:3} are satisfied with $u$ replaced by $u_\eps$ and $T$ replaced by $T_\eps$. Note that all the estimates that we established for a solution $u$ of \eqref{eq:HNS} adapt straightforwardly to a solution $u^\eps$ of \eqref{eq:HNS:approx}. The only notable change is the inclusion of the $-\eps \pa_x^2$ term in \eqref{eq:flat} for defining the boundary layer lift $\omega^{\flat,\eps}$. However, since all estimates for $\omega^{\flat,\eps}$ are obtained by performing a Fourier transform in $x$ and using Plancherel to obtain the desired $L^2_x$ bound, this modification is routine (see also~\cite{IgnatovaVicol16} for $\eps$-independent bounds for analytic in $x$ - Sobolev in $y$ solutions of the $\eps$-regularization of the Prandtl system).  Applying   Corollaries~\ref{cor:omega:Gevrey},~\ref{cor:dy:omega:Gevrey},~\ref{cor:dy:dy:omega}, and Prosition~\ref{prop:maximum:principle} at positive $\eps$, we see that there exists $T>0$ independent of $\eps$, such that for all $t \in [0,\min(T_\eps,T)]$, the conditions  \eqref{eq:omi:ass:1}-\eqref{eq:omi:ass:2}-\eqref{eq:omi:ass:3} still hold with $M$ replaced by $\frac{M}{2}$, and $\delta_0$ replaced by $2\delta_0$. If $T_\eps < T$, then one has necessarily $T_\eps = T_{\eps,*}$, otherwise by continuity the inequalities \eqref{eq:omi:ass:1}-\eqref{eq:omi:ass:2}-\eqref{eq:omi:ass:3} would be satisfied beyond $T_\eps$. But then there  is a contradiction between \eqref{eq:blowup} and the first half of  \eqref{eq:omi:ass:1}. Hence, $T_\eps \ge T$, and so $T_{\eps,*} \ge T$. 

\mspace
We have just shown that the approximations $u_\eps$ are all defined on a time interval independent of $\eps$, and satisfy uniform Gevrey bounds on it. This allows to let $\eps$ go to zero, and conclude by standard compactness arguments to the existence of a solution.

\mspace
For the uniqueness of solutions, the equation obeyed by the difference is basically a linearized version of the equation, very similar to the equation obeyed by $\dot \omega$. Then an estimate similar to the one from Proposition~\ref{prop:omega:dot}, gives the good estimate for the difference of two solutions, implying uniqueness.

\appendix
\section{Proof of Lemma \ref{lemma_Fourier}} \label{appendA}
To prove the first item, we adapt arguments of \cite[pages 1805-1807]{Kuramoto}. We fix $x \in \T$, $y > 0$, and drop them from notations. We write 
$$\hat{\bar{\omega}}_j^\flat(\eta) =  \hat{f}_j(\zeta) \, g_j(\zeta), \quad g_j(\zeta) = \frac{1}{2 - \sqrt{\beta (j+1) + i \zeta}} e^{- y \sqrt{\beta (j+1)+i\zeta}}. $$
Clearly, as $f_j = 0$ for $t < 0$ and belongs to $L^1(\R)$,
$$  \hat{f}_j(\zeta) =  \int_{\R_+} f_j(t) e^{-i \zeta t} dt $$ 
is holomorphic for $\Im \zeta < 0$, and continuous for $\Im \zeta \le 0$. Moreover, 
\begin{equation} \label{uniform_fj}
\lim_{\Im \zeta  \rightarrow +\infty}  \hat{f}_j(\zeta) = 0 \: \text{ uniformly for $\Re \zeta \in \R$}, \quad \lim_{\Re \zeta  \rightarrow \pm\infty}   \hat{f}_j(\zeta) = 0  \: \text{ uniformly for  $\Im \zeta \le 0$}. 
\end{equation}
The first limit follows directly from the inequality 
$$ | \hat{f}_j(\zeta) | \le   \int_{\R_+}  |f_j(t)| e^{-\Im \zeta t} dt $$ 
and the dominated convergence theorem. The second limit follows from a close look at Riemann-Lebesgue's lemma: given $\eps > 0$, and some  $f^\eps_j \in C^1_c(\R_+)$ with $\int_{\R_+} |f_j - f_j^\eps| \le \eps$, we get
\begin{align*}
 |\hat{f}_j(\zeta)|  & \le \int_{\R+} | f_j - f_j^\eps|  \: + \:   | \int_{\R+} f_j^\eps(t)   e^{-i \zeta t} dt |  \\
 & \le  \eps + \frac{M_\eps}{|\Re \zeta|}
 \end{align*}
where the second bound follows from an integration by part of the second integral. 

Obviously,  $g_j$ is also holomorphic in $\Im \zeta < 0$,  continuous  over $\Im \zeta \le 0$, with bound 
\begin{equation} \label{uniform_gj}
 |g_j(\zeta)| \le \frac{1}{\beta-2}e^{-\sqrt{|\zeta|} y}, 
 \end{equation}
see \eqref{lowerbound_Fourier}. We finally  apply the Cauchy formula: for any  $t < 0$, for any $\mu > 0$
\begin{align*}
 \overline{\omega}_j^\flat(t)    = & \lim_{s \rightarrow +\infty} \: \frac{1}{2\pi} \int_{-s}^s \hat{f}_j(\zeta) \, g_j(\zeta)  e^{i \zeta t}\, d\zeta \\
  =  &  - \lim_{s \rightarrow +\infty}   \: \frac{1}{2\pi}  \biggl(   \int_{[-s,s]-i\mu}  \hat{f}_j(\zeta) \, g_j(\zeta)  e^{i \zeta t} \, d\zeta  + \int_{[s,s-i\mu]} \hat{f}_j(\zeta) \, g_j(\zeta)  e^{i \zeta t} \, d\zeta \\
 & + \int_{[-s-i\mu, -s]} \hat{f}_j(\zeta) \, g_j(\zeta)  e^{i \zeta t} \, d\zeta \biggr) 
\end{align*}
As $t < 0$, taking into account the first limit in \eqref{uniform_fj}, the first integral at the right-hand side goes to zero when $\mu \rightarrow +\infty$, while the two other integrals over the vertical segments converge to the integrals over the vertical half-lines: 
\begin{align*}
 \overline{\omega}_j^\flat(t)    = & \lim_{s \rightarrow +\infty} \frac{1}{2\pi}  \biggl(  \int_{[s,s-i\infty]} \hat{f}_j(\zeta) \, g_j(\zeta)  e^{i \zeta t} \, d\zeta  + \int_{[-s-i\infty, -s]} \hat{f}_j(\zeta) \, g_j(\zeta)  e^{i \zeta t} \, d\zeta \biggr) \\
  = & \lim_{s \rightarrow +\infty} \frac{1}{2\pi}     \biggl(  \int_{[0,-i\infty]} \hat{f}_j(s+\zeta) \, g_j(s+\zeta)  e^{i (s+\zeta) t} \, d\zeta  + \int_{[-i\infty, 0]} \hat{f}_j(-s+\zeta) \, g_j(-s+\zeta)  e^{i (-s + \zeta) t} \, d\zeta \biggr)
 \end{align*} 
Using the second limit in \eqref{uniform_fj} and the bound \eqref{uniform_gj}, we can conclude that the limit at the right-hand side is zero thanks to the dominated convergence theorem. 

\medskip 
To prove the second item of the lemma, we remark from formula \eqref{Fourier_omega}   that 
$$  (1+|\zeta|)^{3/4} \hat{\omega}_j^\flat  \in L^2_\zeta(\R, L^2_y(\R_+, H^k_x(\T))), \quad (1+|\zeta|)^{1/4} \hat{\omega}_j^\flat  \in L^2_\zeta(\R, H^1_y(\R_+, H^k_x(\T))), \quad \forall k$$ 
using the smoothness of $\hat{f}_j$ with respect to $x$. We deduce that 
\begin{equation} \label{regularity}
\overline{\omega}_j^\flat \in H^{3/4}_t(\R, L^2_y(\R_+, H^k_x(\T))), \quad \overline{\omega}_j^\flat \in H^{1/4}_t(\R, H^1_y(\R_+, H^k_x(\T)), \quad \forall k.
\end{equation}
Moreover, using again \eqref{Fourier_omega} and Plancherel in time, we get  that:   for any $\varphi = \varphi(t,x,y)$ smooth and fastly decreasing as $t \rightarrow \pm \infty$ and $y \rightarrow +\infty$, 
\begin{equation*}
\int_{\R \times \R_+ \times \T}   \overline{\omega}_j^\flat  \, (\beta (j+1) - \pa_t)  \varphi  \: + \:  \int_{\R \times \R_+ \times \T}  \pa_y    \overline{\omega}_j^\flat    \, \pa_y \varphi  - \int_{\R \times \T} (2 \overline{\omega}_j^\flat\vert_{y=0} + f_j) \, \varphi\vert_{y=0}  = 0.
\end{equation*}
If we take $\varphi$ with support in time included in $(-\infty, T)$, taking into account that  $\overline{\omega}_j^\flat$ is zero for negative times, we end up with 
\begin{equation*}
\int_{(0,T) \times \R_+ \times \T}   \overline{\omega}_j^\flat  \, (\beta (j+1) - \pa_t)  \varphi  \: + \:  \int_{(0,T) \times \R_+ \times \T}  \pa_y    \overline{\omega}_j^\flat    \, \pa_y \varphi  - \int_{(0,T) \times \T} (2 \overline{\omega}_j^\flat\vert_{y=0} + \frac{M_j}{M_{j+1}}h_{j+1}) \, \varphi\vert_{y=0} =0.
\end{equation*}
We recognize the weak formulation of system \eqref{eq:omflatj1}-\eqref{eq:omflatj2}-\eqref{eq:omflatj3}. The identity $\overline{\omega}_j^\flat = \omega_j^\flat$ over $(0,T)$ follows from the uniqueness of solutions to this system (for example in the regularity class given by \eqref{regularity}).

\bibliographystyle{abbrv}
\bibliography{DavidBib,VladBib}

\begin{thebibliography}{10}

\bibitem{Brenier99}
Y.~Brenier.
\newblock Homogeneous hydrostatic flows with convex velocity profiles.
\newblock {\em Nonlinearity}, 12(3):495--512, 1999.

\bibitem{Brenier03}
Y.~Brenier.
\newblock Remarks on the derivation of the hydrostatic {E}uler equations.
\newblock {\em Bull. Sci. Math.}, 127(7):585--595, 2003.

\bibitem{BreschGuillenMadmoudiRodriguez03}
D.~Bresch, F.~Guill{\'e}n-Gonz{\'a}lez, N.~Masmoudi, and M.~Rodr\'\i~guez
  Bellido.
\newblock On the uniqueness of weak solutions of the two-dimensional primitive
  equations.
\newblock {\em Differential Integral Equations}, 16(1):77--94, 2003.

\bibitem{BreschKazhikhovLemoine04}
D.~Bresch, A.~Kazhikhov, and J.~Lemoine.
\newblock On the two-dimensional hydrostatic {N}avier-{S}tokes equations.
\newblock {\em SIAM J. Math. Anal.}, 36(3):796--814, 2004/05.

\bibitem{CaoIbrahimNakanishiTiti15}
C.~Cao, S.~Ibrahim, K.~Nakanishi, and E.~Titi.
\newblock Finite-time blowup for the inviscid {P}rimitive equations of oceanic
  and atmospheric dynamics.
\newblock {\em Comm. Math. Phys.}, 337(2):473--482, 2015.

\bibitem{CaoLiTiti16}
C.~Cao, J.~Li, and E.~Titi.
\newblock Global well-posedness of the three-dimensional primitive equations
  with only horizontal viscosity and diffusion.
\newblock {\em Comm. Pure Appl. Math.}, 69(8):1492--1531, 2016.

\bibitem{CaoLiTiti17}
C.~Cao, J.~Li, and E.~Titi.
\newblock Strong solutions to the 3{D} primitive equations with only horizontal
  dissipation: near {$H^1$} initial data.
\newblock {\em J. Funct. Anal.}, 272(11):4606--4641, 2017.

\bibitem{CaoTiti07}
C.~Cao and E.~Titi.
\newblock Global well-posedness of the three-dimensional viscous primitive
  equations of large scale ocean and atmosphere dynamics.
\newblock {\em Ann. of Math. (2)}, 166(1):245--267, 2007.

\bibitem{DalibardMasmoudi18}
A.-L. Dalibard and N.~Masmoudi.
\newblock Separation for the stationary {P}randtl equation.
\newblock {\em arXiv:1802.04039}, 2018.

\bibitem{EEngquist97}
W.~E and B.~Engquist.
\newblock Blowup of solutions of the unsteady {P}randtl's equation.
\newblock {\em Comm. Pure Appl. Math.}, 50(12):1287--1293, 1997.

\bibitem{Kuramoto}
B.~Fernandez, D.~G\'erard-Varet, and G.~Giacomin.
\newblock Landau damping in the {K}uramoto model.
\newblock {\em Ann. Henri Poincar\'e}, 17(7):1793--1823, 2016.

\bibitem{GerardVaretDormy10}
D.~G{{\'e}}rard-Varet and E.~Dormy.
\newblock On the ill-posedness of the {P}randtl equation.
\newblock {\em J. Amer. Math. Soc.}, 23(2):591--609, 2010.

\bibitem{GerardVaretMaekawaMasmoudi16}
D.~G{{\'e}}rard-Varet, Y.~Maekawa, and N.~Masmoudi.
\newblock {G}evrey stability of {P}randtl expansions for {2D}
  {N}avier-{S}tokes.
\newblock {\em arXiv:1607.06434}, 2016.

\bibitem{GerardVaretMasmoudi13}
D.~G{{\'e}}rard-Varet and N.~Masmoudi.
\newblock Well-posedness for the {P}randtl system without analyticity or
  monotonicity.
\newblock {\em Ann. Sci. \'{E}c. Norm. Sup\'er. (4)}, 48(6):1273--1325, 2015.

\bibitem{GerardVaretNguyen12}
D.~G\'{e}rard-Varet and T.~Nguyen.
\newblock Remarks on the ill-posedness of the {P}randtl equation.
\newblock {\em Asymptotic Analysis}, 77:71--88, 2012.

\bibitem{Grenier99}
E.~Grenier.
\newblock On the derivation of homogeneous hydrostatic equations.
\newblock {\em M2AN Math. Model. Numer. Anal.}, 33(5):965--970, 1999.

\bibitem{Grenier}
E.~Grenier.
\newblock On the stability of boundary layers of incompressible {E}uler
  equations.
\newblock {\em J. Differential Equations}, 164(1):180--222, 2000.

\bibitem{GrGuNg}
E.~Grenier, Y.~Guo, and T.~T. Nguyen.
\newblock Spectral instability of general symmetric shear flows in a
  two-dimensional channel.
\newblock {\em Adv. Math.}, 292:52--110, 2016.

\bibitem{HongHunter03}
L.~Hong and J.~Hunter.
\newblock Singularity formation and instability in the unsteady inviscid and
  viscous {P}randtl equations.
\newblock {\em Commun. Math. Sci.}, 1(2):293--316, 2003.

\bibitem{IgnatovaVicol16}
M.~Ignatova and V.~Vicol.
\newblock Almost global existence for the {P}randtl boundary layer equations.
\newblock {\em Arch. Ration. Mech. Anal.}, 220(2):809--848, 2016.

\bibitem{Kobelkov07}
G.~Kobelkov.
\newblock Existence of a solution ``in the large'' for ocean dynamics
  equations.
\newblock {\em J. Math. Fluid Mech.}, 9(4):588--610, 2007.

\bibitem{KukavicaLombardoSammartino16}
I.~Kukavica, M.~Lombardo, and M.~Sammartino.
\newblock Zero viscosity limit for analytic solutions of the primitive
  equations.
\newblock {\em Arch. Ration. Mech. Anal.}, 222(1):15--45, 2016.

\bibitem{KukavicaMasmoudiVicolWong14}
I.~Kukavica, N.~Masmoudi, V.~Vicol, and T.~Wong.
\newblock On the local well-posedness of the {P}randtl and the hydrostatic
  {E}uler equations with multiple monotonicity regions.
\newblock {\em SIAM J. Math. Anal.}, 46(6):3865--3890, 2014.

\bibitem{KukavicaTemamVicolZiane11}
I.~Kukavica, R.~Temam, V.~Vicol, and M.~Ziane.
\newblock Local existence and uniqueness for the hydrostatic {E}uler equations
  on a bounded domain.
\newblock {\em J. Differential Equations}, 250(3):1719--1746, 2011.

\bibitem{KukavicaVicol13}
I.~Kukavica and V.~Vicol.
\newblock On the local existence of analytic solutions to the {P}randtl
  boundary layer equations.
\newblock {\em Commun. Math. Sci.}, 11(1):269--292, 2013.

\bibitem{KukavicaVicolWang17}
I.~Kukavica, V.~Vicol, and F.~Wang.
\newblock The van {D}ommelen and {S}hen singularity in the {P}randtl equations.
\newblock {\em Adv. Math.}, 307:288--311, 2017.

\bibitem{KukavicaZiane07}
I.~Kukavica and M.~Ziane.
\newblock On the regularity of the primitive equations of the ocean.
\newblock {\em Nonlinearity}, 20(12):2739--2753, 2007.

\bibitem{KukavicaZiane08}
I.~Kukavica and M.~Ziane.
\newblock Uniform gradient bounds for the primitive equations of the ocean.
\newblock {\em Differential Integral Equations}, 21(9-10):837--849, 2008.

\bibitem{LagreeLorthois05}
P.-Y. Lagr\'ee and S.~Lorthois.
\newblock The {RNS/Prandtl} equations and their link with other asymptotic
  descriptions: application to the wall shear stress scaling in a constricted
  pipe.
\newblock {\em Int. J. Eng. Sci.}, 43(3-4):352--378, 2005.

\bibitem{LionsTemamWang92b}
J.-L. Lions, R.~Temam, and S.~Wang.
\newblock New formulations of the primitive equations of atmosphere and
  applications.
\newblock {\em Nonlinearity}, 5(2):237--288, 1992.

\bibitem{LionsTemamWang92a}
J.-L. Lions, R.~Temam, and S.~Wang.
\newblock On the equations of the large-scale ocean.
\newblock {\em Nonlinearity}, 5(5):1007--1053, 1992.

\bibitem{MasmoudiWong12}
N.~Masmoudi and T.~Wong.
\newblock On the ${H}^s$ theory of hydrostatic {E}uler equations.
\newblock {\em Arch. Ration. Mech. Anal.}, 204(1):231--271, 2012.

\bibitem{MasmoudiWong15}
N.~Masmoudi and T.~Wong.
\newblock Local-in-time existence and uniqueness of solutions to the {P}randtl
  equations by energy methods.
\newblock {\em Comm. Pure Appl. Math.}, 68(10):1683--1741, 2015.

\bibitem{Oleinik66}
O.~Oleinik.
\newblock On the mathematical theory of boundary layer for an unsteady flow of
  incompressible fluid.
\newblock {\em J. Appl. Math. Mech.}, 30:951--974 (1967), 1966.

\bibitem{Petcu04}
M.~Petcu.
\newblock Gevrey class regularity for the primitive equations in space
  dimension 2.
\newblock {\em Asymptot. Anal.}, 39(1):1--13, 2004.

\bibitem{PetcuTemamWirosoetisno04}
M.~Petcu, R.~Temam, and D.~Wirosoetisno.
\newblock Existence and regularity results for the primitive equations in two
  space dimensions.
\newblock {\em Commun. Pure Appl. Anal.}, 3(1):115--131, 2004.

\bibitem{PetcuTemamZiane09}
M.~Petcu, R.~Temam, and M.~Ziane.
\newblock Some mathematical problems in geophysical fluid dynamics.
\newblock In {\em Handbook of numerical analysis. {V}ol. {XIV}. {S}pecial
  volume: computational methods for the atmosphere and the oceans}, volume~14
  of {\em Handb. Numer. Anal.}, pages 577--750. Elsevier/North-Holland,
  Amsterdam, 2009.

\bibitem{Renardy09}
M.~Renardy.
\newblock Ill-posedness of the hydrostatic {E}uler and {N}avier-{S}tokes
  equations.
\newblock {\em Arch. Ration. Mech. Anal.}, 194(3):877--886, 2009.

\bibitem{SammartinoCaflisch98a}
M.~Sammartino and R.~Caflisch.
\newblock Zero viscosity limit for analytic solutions, of the {N}avier-{S}tokes
  equation on a half-space. {I}. {E}xistence for {E}uler and {P}randtl
  equations.
\newblock {\em Comm. Math. Phys.}, 192(2):433--461, 1998.

\bibitem{TemamZiane04}
R.~Temam and M.~Ziane.
\newblock Some mathematical problems in geophysical fluid dynamics.
\newblock In {\em Handbook of mathematical fluid dynamics. Vol. III}, pages
  535--657. North-Holland, Amsterdam, 2004.

\bibitem{Wong15}
T.~Wong.
\newblock Blowup of solutions of the hydrostatic {E}uler equations.
\newblock {\em Proc. Amer. Math. Soc.}, 143(3):1119--1125, 2015.

\bibitem{XinZhang04}
Z.~Xin and L.~Zhang.
\newblock On the global existence of solutions to the {P}randtl's system.
\newblock {\em Adv. Math.}, 181(1):88--133, 2004.

\bibitem{Ziane97b}
M.~Ziane.
\newblock Regularity results for the stationary primitive equations of the
  atmosphere and the ocean.
\newblock {\em Nonlinear Anal.}, 28(2):289--313, 1997.

\end{thebibliography}
% Pour Nader: j'ai change \bibliography{../biblio} en  \bibliography{biblio}. Si tu veux recompiler correctement, il faut remettre  \bibliography{../biblio} 

\end{document}